\documentclass[a4paper,reqno]{amsart}

\usepackage[french,british]{babel}
\usepackage{amssymb,amsmath,amsthm,amscd,mathrsfs,graphicx,color,multicol,booktabs,bm,manfnt}
\usepackage[cmtip,all,matrix,arrow,tips,curve]{xy}

\usepackage{mathrsfs}
\usepackage[colorlinks=true,citecolor=blue, urlcolor=blue, linkcolor=blue, pagebackref]{hyperref}

\usepackage{mathabx}

\usepackage{manfnt}
\usepackage{centernot}   

\newcommand{\CC}{{\mathbb C}}

\newcommand{\cD}{{\mathscr D}}
\newcommand{\cE}{{\mathscr E}}
\newcommand{\cF}{{\mathscr F}}
\newcommand{\cG}{{\mathscr G}}
\newcommand{\cH}{{\mathscr H}}

\newcommand{\cJ}{{\mathscr J}}
\newcommand{\cK}{{\mathscr K}}
\newcommand{\cL}{{\mathscr L}}
\newcommand{\cM}{{\mathscr M}}
\newcommand{\cN}{{\mathscr N}}

\newcommand{\cO}{{\mathscr O}}
  
\newcommand{\cQ}{{\mathscr Q}}

\newcommand{\cS}{{\mathscr S}}
\newcommand{\cT}{{\mathscr T}}
\newcommand{\cU}{{\mathscr U}}
\newcommand{\cV}{{\mathscr V}} 

\newcommand{\cX}{{\mathscr X}} 
\newcommand{\cY}{{\mathscr Y}}

\newcommand{\dra}{\dashrightarrow}

\newcommand{\hra}{\hookrightarrow}

\newcommand{\la}{\langle}

\newcommand{\lra}{\longrightarrow}

\newcommand{\NN}{{\mathbb N}}
\newcommand{\ov}{\overline}
\newcommand{\PP}{{\mathbb P}}
\newcommand{\QQ}{{\mathbb Q}}
\newcommand{\ra}{\rangle}

\newcommand{\wh}{\widehat}
\newcommand{\wt}{\widetilde}
\newcommand{\ZZ}{{\mathbb Z}}

\theoremstyle{plain}

\newtheorem{thm}{Theorem}[section]
\newtheorem*{thm*}{Theorem}

\newtheorem{clm}[thm]{Claim}

\newtheorem{crl}[thm]{Corollary}

\newtheorem*{hyp*}{Hypothesis}
\newtheorem{lmm}[thm]{Lemma}
\newtheorem{prp}[thm]{Proposition}
\newtheorem{prp-dfn}[thm]{Proposition-Definition}

\theoremstyle{definition}

\newtheorem{dfn}[thm]{Definition}

\theoremstyle{remark}

\newtheorem*{qst*}{Main Question}
\newtheorem{rmk}[thm]{Remark}

\DeclareMathOperator{\Aut}{Aut}

\DeclareMathOperator{\ch}{ch}
\DeclareMathOperator{\CH}{CH}

\DeclareMathOperator{\cl}{cl}

\DeclareMathOperator{\Def}{Def}

\DeclareMathOperator{\disc}{disc}
\DeclareMathOperator{\divisore}{div}

\DeclareMathOperator{\End}{End}

\DeclareMathOperator{\Hom}{Hom}

\DeclareMathOperator{\im}{Im}

\DeclareMathOperator{\NL}{\mathit{NL}}
\DeclareMathOperator{\NS}{NS}

\DeclareMathOperator{\Pic}{Pic}

\DeclareMathOperator{\pr}{pr}

\DeclareMathOperator{\rk}{rk}

\DeclareMathOperator{\sing}{sing}

\DeclareMathOperator{\Sym}{Sym}
\DeclareMathOperator{\Td}{Td}

\usepackage{multirow}

 \makeindex
 
\begin{document}
 \title{Rigid stable vector bundles on hyperk\"ahler varieties of type $K3^{[n]}$}
 \author{Kieran G. O'Grady}
 \address{Dipartimento di Matematica, 
 Sapienza Universit\`a di Roma,
 P.le A.~Moro 5,
 00185 Roma - ITALIA}
 \email{ogrady@mat.uniroma1.it}
\date{\today}
\thanks{Partially supported by PRIN 2017YRA3LK}
\begin{abstract}
We prove existence and unicity of slope stable vector bundles   on a general polarized hyperk\"ahler (HK) variety of type $K3^{[n]}$ with certain discrete invariants, provided the rank and the first two Chern classes of the vector bundle satisfy certain equalities. The latter hypotheses at first glance appear to be quite restrictive, but in fact we might have listed almost all slope stable rigid projectively hyperholomorphic vector bundles on  polarized HK varieties of type $K3^{[n]}$ with $20$ moduli.  
\end{abstract}

  \maketitle
\bibliographystyle{amsalpha}
\section{Introduction}\label{sec:intro}
\subsection{Background}
\setcounter{equation}{0}
 A prominent r\^ole in the theory of $K3$ surfaces is played by spherical  (i.e.~rigid and simple) vector bundles. In~\cite{ogfascimod}  we have proved existence and uniqueness results for stable vector bundles on  general polarized hyperk\"ahler (HK) variety of type 
$K3^{[2]}$ with certain discrete invariants (of the polarization and of the vector bundle).   In the present paper we show that 
the main result in~\cite{ogfascimod}  extends to HK varieties of type $K3^{[n]}$ of arbitrary (even) dimension. More precisely, we prove  that for certain choices of rank and first two Chern classes  on a polarized HK variety $(X,h)$  of type $K3^{[n]}$, 
 there exists one and only one 
 stable vector bundle with the assigned  rank and first two Chern classes  provided the moduli point of $(X,h)$ is a general point of a certain irreducible component of the relevant moduli space of polarized $HK$ varietes.
 
We like to think of this result as evidence in favour of the following slogan: \emph{stable vector bundles on higher dimensional HK manifolds behave as well as  stable sheaves on $K3$ surfaces, provided one restricts to (stable) vector bundles whose projectivization extends to all small deformations of the base HK manifold (i.e.~projectively hyperhomolomorphic vector bundles)}.
\subsection{The main result}\label{subsec:lalaland}
\setcounter{equation}{0}
Let $\cK^d_e(2n)$ be the moduli space of polarized HK varieties  of type $K3^{[n]}$ of degree $e$ and divisibility $d$. Thus 
$\cK^d_e(2n)$ parametrizes isomorphism classes of couples $(X,h)$ where $X$ is a HK manifold  of type 
$K3^{[n]}$, $h\in\NS(X)$ is the class of an ample divisor class  (we assume that $h$ is primitive) such that $q_X(h)=e$ and 
$\divisore(h)=d$, where $q_X$ is the Beauville-Bogomolov-Fujiki (BBF) quadratic form of $X$ and 
$\divisore(h)$ is the divisibility of $h$, i.e.~the positive generator of $q_X(h,H^2(X;\ZZ))$. Note that 
$\divisore(h)$ divides $2(n-1)$. It is known under which hypotheses
$\cK^d_e(2n)$ is not empty. If that is the case, then it is a quasi-projective variety (not necessarily irreducible) of pure dimension $20$.  

We recall that if $\cF$ is a (coherent) sheaf on a complex smooth variety the \emph{discriminant of $\cF$} is defined to be the Betti cohomology class
\begin{equation}
\Delta(\cF):=2r c_2(\cF)-(r-1) c_1(\cF)^2=-2r \ch_2(\cF)+\ch_1(\cF)^2.
\end{equation}
\begin{thm}\label{thm:unicita}
 Let $n,r_0,g,l,e\in\NN_{+}$, with $n\ge 2$, and let
\begin{equation*}
\ov{e}:=
\begin{cases}
\text{$e$ if $r_0$ is even,} \\
 \text{$4e$  if $r_0$ is odd.}
\end{cases}
\end{equation*}
 Assume that 
\begin{equation}\label{adancona}
\text{$g$ divides}
\begin{cases}
\text{$(r_0-1)$ if $r_0$ is even,} \\
 \text{$(r_0-1)/2$  if $r_0$ is odd,}
\end{cases}
\end{equation}
that
\begin{equation}\label{nicoletta}
 l|(n-1),\quad \gcd\{l, r_0\}=1,\quad \gcd\left\{l, \frac{r_0-1}{g}\right\}=1,
\end{equation}
 that
\begin{equation}\label{nuvolari}
g^2 \cdot \ov{e} +2(n-1)(r_0-1)^2+8 \equiv 0 \pmod{8r_0},
\end{equation}
and that
\begin{equation}\label{cantarini}
\ov{e}+\frac{2(n-1)(r_0-1)^2}{g^2}\equiv 0\pmod{8l^2}.
\end{equation}
Let $i\in\{1,2\}$ be such that $i\equiv r_0\pmod{2}$.
Then  
there exists an irreducible component of $\cK^{il}_e(2n)$, denoted by $\cK^{il}_e(2n)^{\rm good}$ (see Definition~\ref{dfn:compbuona}) such that for a  general  
 $[(X,h)]\in \cK^{il}_e(2n)^{\rm good}$  there exists one and only one (up to isomorphism) $h$ slope-stable vector bundle $\cE$ on $X$ such that
 \begin{equation}\label{chernvbe}
r(\cE)=r_0^n,\quad c_1(\cE)=\frac{g\cdot r_0^{n-1}}{i} h,\quad \Delta(\cE)  =  \frac{r_0^{2n-2}(r_0^2-1)}{12}c_2(X). 
\end{equation}
Moreover for such a vector bundle $H^p(X,End^0(\cE))=0$ for all $p$. 
\end{thm}
\subsection{Comments}
\setcounter{equation}{0}
\subsubsection{Special cases}
Let $g=l=1$ in Theorem~\ref{thm:unicita}. Then the hypotheses reduce to the following congruences:
\begin{equation}\label{econ}
e\equiv
\begin{cases}
4(n-1)r_0-2n-6 \pmod{8r_0} & \text{if $r_0\equiv 0 \pmod{4}$,} \\
\frac{1}{2}((n-1)r_0-n-3) \pmod{2r_0} & \text{if $r_0\equiv 1 \pmod{4}$,} \\
-2n-6 \pmod{8r_0}  & \text{if $r_0\equiv 2 \pmod{4}$,} \\
-\frac{1}{2}((n-1)r_0+n+3) \pmod{2r_0}  & \text{if $r_0\equiv 3 \pmod{4}$.}
\end{cases}
\end{equation}
In particular  for  $n=2$ (and $g=l=1$) Theorem~\ref{thm:unicita} reduces to  the main result  in~\cite{ogfascimod}. 
Note also that in this case $\cK^{il}_e(2n)$ is irreducible by~\cite[Theorem~3.5]{debarre-surveyhk}, and hence 
$\cK^{il}_e(2n)^{\rm good}=\cK^{il}_e(2n)$.

\subsubsection{Rank and Chern classes}
The choice of rank and first two Chern classes in Theorem~\ref{thm:unicita} is  not as special  as one would think. 
Let us first consider a  rigid stable vector bundle $\cE$ on a polarized $K3$ surface $(S,h)$ of rank $r$ with $c_1(\cE)=ah$. Since $\chi(S,End\,\cE)=2$,  we have
\begin{equation}\label{disckappa}
2r c_2(\cE)-(r-1) a^2 h^2=\Delta(\cE)=2(r^2-1).
\end{equation}
It follows that $\gcd\{r,a\}=1$. The following result is a (weak) extension  to higher dimensions.
\begin{prp}\label{prp:rangociuno}
Let $(X,h)$ be a polarized HK variety of type $K3^{[n]}$, and let $\cF$  be a slope stable vector bundle on $X$. Suppose  that $c_1(\cF)=ah$ and that the natural morphism $\Def(X,\cF)\to\Def(X,h)$ is surjective. Then
\begin{equation}\label{thatsamore}
r(\cF)=r_0^n m,\quad a\cdot\divisore(h)=r_0^{n-1} m b_0',
\end{equation}
where $r_0,m,b_0'$ are integers, and $\gcd\{r_0,b_0'\}=1$.
\end{prp}
Note that in  Proposition~\ref{prp:rangociuno} we do not assume that $\cF$  is rigid. There are examples of slope stable projectively hyperholomorphic  vector bundles  for which $m>1$ which are not rigid
see~\cite{markmod,bottini-og10} and~\cite{fatighentipotenza}. The tangent vector bundle of a HK 
manifold of type $K3^{[2]}$  is an example of rigid slope stable projectively hyperholomorphic  vector bundle with $m>1$ (in fact $m=4$), see~\cite{gavran-tan-rigid}. 
There are no other examples of the latter type that I am aware of.

The stable vector bundles in Theorem~\ref{thm:unicita}, together with those obtained by tensoring with powers of the polarization, cover many of the choices of rank and first Chern class with  $m=1$ which are a priori possible according to Proposition~\ref{prp:rangociuno}. 

Regarding the discriminant  of the vector bundle(s) $\cE$  in Theorem~\ref{thm:unicita} we note the following. First, the formula for the discriminant
 in~\eqref{chernvbe} for $n=1$ is exactly the formula in~\eqref{disckappa}. Next, since $-\Delta(\cE)$ is the second Chern class of the push-forward to $X$ of the relative tangent bundle of $\PP(\cE)\to X$, and 
 $\PP(\cE)$ extends to all small deformations of $X$ (because $H^2(X,End^0(\cE))=0$), the cohomology class
  $\Delta(\cE)$ remains of type $(2,2)$ for all deformations of $X$. It follows that
 $\Delta(\cE)$ is a  linear combination of $c_2(X)$ and $q_X^{\vee}$, see  the main result in~\cite{zhang-char-form}. 
 If $n\in\{2,3\}$ then $c_2(X)$ and $q_X^{\vee}$ are linearly dependent, and hence it follows (without doing any computation) that $\Delta(\cE)$ is  a multiple of $c_2(X)$.  If $n>3$ then $c_2(X)$ and $q_X^{\vee}$ are linearly independent, hence there is no  \lq\lq a priori\rq\rq\  reason why  $\Delta(\cE)$ should be   a multiple of $c_2(X)$. In fact I know of no examples of stable projectively hyperholomorphic  vector bundles on HK varieties of type $K3^{[n]}$ whose discriminant is not a multiple of the second Chern class.

The vector bundles $\cE$  in Theorem~\ref{thm:unicita} are atomic, in fact they are in the $\cO_X$-orbit (see~\cite{markmod}) and  hence the Beckmann-Markman extended Mukai vector 
$\wt{v}(\cE)$ (see~\cite[Theorem 6.13]{markmod}, \cite[Definition 4.16]{beckmann-ext-mukai},  
\cite[Definition 1.1]{beckmann-atomic}) is determined uniquely (if we require that its first entry equals $r(\cE)$), in particular $\wt{q}(\wt{v}(\cE))=-(n+3)r_0^{2n-2}/2$ 
by~\cite[Lemma~4.8]{beckmann-ext-mukai}. By the formula relating the projection of the discriminant 
$\Delta(\cE)$ on the Verbitsky subalgebra and the square 
$\wt{q}(\wt{v}(\cE))$ in~\cite[Proposition~3.11]{bottini-og10}, we get that if $\Delta(\cE)$ is  a multiple of $c_2(X)$,  then it is given by the formula in~\eqref{chernvbe}.

The natural question to ask is the following: are we close to having listed \emph{all slope stable rigid vector bundles on a  polarized HK variety of type $K3^{[n]}$ with $20$ moduli}?

\subsubsection{Projective bundles}
Let $X$ and $\cE$ be as in Theorem~\ref{thm:unicita}. The projectivization $\PP(\cE)$  extends (uniquely) to a projective bundle on all (small) deformations of $X$ because $H^2(X,End^0(\cE))=0$. Actually 
Markman~\cite[Theorem 1.4]{markmod} shows that $\PP(\cE)$  extends  to a projective bundle on all 
deformations of $X$ (it is projectively hyperholomorphic). In fact the (possibly twisted) locally free sheaf $E$ on $S^{[n]}$ appearing in Markman loc.~cit.~is obtained by deforming the vector bundle 
$\cF[n]^{+}$ associated to a spherical vector bundle $\cF$ on $S$, see Definition~3.1 (or Definition~5.1 in~\cite{ogfascimod}), and likewise the vector bundles in Theorem~1.1 are obtained by deforming 
$\cF[n]^{+}$. 
Theorem~\ref{thm:unicita} should provide a uniqueness result for stable projective bundles 
$\mathbf P$ of dimension $(r_0^n-1)$ with characteristic class given by the third equality in~\eqref{chernvbe} (i.e.~$-c_2(\Theta_{{\mathbf P}/X})$, where $\Theta_{{\mathbf P}/X}$ is the relative tangent bundle of ${\mathbf P}\to X$)
on a  general HK manifold of type $K3^{[n]}$. In order to turn this into a precise statement one would need to specify with respect to which K\"ahler classes the projective bundle is supposed to be stable. 
The zoo of conditions in Theorem~\ref{thm:unicita} would then correspond to the cases in which the projective bundle is the projectivization of a vector bundle, i.e.~to the vanishing of the relevant Brauer class. 

\subsubsection{Franchetta property}
Let $\cU^{il}_e(2n)\subset \cK^{il}_e(2n)$ be an open non empty subset with the property that there exists one and only one stable vector bundle $\cE$ 
on $[(X,h)]\in \cU^{il}_e(2n)$ such that the equations in~\eqref{chernvbe} hold, and let $\cX\to \cU^{il}_e(2n)$ be the tautological family of HK (polarized) varieties (we might need to pass to the moduli stack). By Theorem A.5 in~\cite{mukvb}, there exists a quasi tautological vector bundle $\mathsf E$ on $\cX$, i.e.~a vector bundle whose restriction to a fiber $(X,h)$ of $\cX\to \cU^{il}_e(2n)$ is isomorphic to $\cE^{\oplus d}$ for some $d>0$, where $\cE$ is the vector bundle of Theorem~\ref{thm:unicita}. If $[(X,h)]\in \cU^{il}_e(2n)$ the Generalized Franchetta conjecture, see~\cite{fulatvial:genfranchetta1}, predicts that the restriction to $\CH^2(X)_{\QQ}$  of  
$\ch_2({\mathsf E})\in\CH(\cX)_{\QQ}$ is equal to  $-d\frac{r_0^{2n-2}(r_0^2-1)}{12}c_2(X)$. In other words it predicts that the third equality in~\eqref{chernvbe} holds at the level of (rational) Chow groups.  In general it is not easy to give a rationally defined algebraic cycle class on a non empty open subset of the moduli stack of polarized HK varieties. Theorem~\ref{thm:unicita} produces such a cycle, and hence it provides
 a good test for the generalized 
Franchetta conjecture.

\subsection{Basic ideas}\label{sec:ideebase}
\setcounter{equation}{0}
The key elements in the proof of the main result are the following. First there is the extension to modular sheaves (defined in~\cite{ogfascimod}) on higher dimensional HK manifolds of the decomposition of the (real) ample cone of a smooth projective surface into open chambers for which slope stability of sheaves with fixed numerical characters does not change, see~\cite{ogfascimod}. 

The second element is the behaviour of modular vector bundles on a Lagrangian HK manifold. If the polarization is very close to the pull-back of an ample line bundle from the base, then the restriction of a slope stable vector bundle  to a general Lagrangian fiber is slope semistable, and if it is slope stable then it is a semi-homogeneous vector bundle, in particular it has no non trivial infinitesimal deformations keeping the determinant fixed. In the reverse direction, if the restriction of a  vector bundle  to a general Lagrangian fiber is slope stable, then the vector bundle is slope stable (provided the polarization is very close to the pull-back of an ample line bundle from the base).

The key element in the proof of existence is a construction discussed in~\cite{ogfascimod} (and in~\cite{markmod}) which associates to a vector bundle 
$\cF$ on a $K3$ surface $S$ a sheaf $\cF[n]^{+}$ on $S^{[n]}$.   The sheaf $\cF[n]^{+}$ is locally free by Haiman's highly non trivial results in~\cite{haiman}. If $\cF$ is a spherical vector bundle then $\End^0(\cF[n]^{+})$  has no non zero cohomology by Bridgeland-King-Reid's derived version of the McKay correspondence. This gives that $\cF[n]^{+}$ extends to all (small) deformations of $(S^{[n]},\det\cF[n]^{+})$, and that the projectivization 
$\PP(\cF[n]^{+})$ extends to all (small) deformations of $S^{[n]}$ (the last result follows from a classical result of Horikawa). We prove slope stability of 
 $\cF[n]^{+}$  in the case of an elliptic $K3$ surface $S$, by using our results on vector bundles on Lagrangian HK manifolds. In fact if $S$ is an elliptic $K3$ surface, then there is a Lagrangian fibration 
 $S^{[n]}\to(\PP^1)^{(n)}\cong\PP^n$, whose general fiber is the product of $n$ fibers of the elliptic fibration. If $\cF$ is a slope stable rigid vector bundle on $S$, then the restriction to an elliptic fiber is slope stable. It follows that the restriction of $\cF[n]^{+}$ to a general  fiber of the Lagrangian fibration 
 $S^{[n]}\to\PP^n$ is slope stable.
  From this one gets that the (unique) extension of  $\cF[n]^{+}$  to a general Lagrangian deformation of $(S^{[n]},\det\cF[n]^{+})$ is slope stable with respect to  $\det\cF[n]^{+}$ (provided we move in a Noether-Lefschetz locus with  high enough  discriminant).
 
Uniqueness of a general slope stable vector bundle with the given numerical invariants is obtained by proving uniqueness for  vector bundles on (polarized) HK varieties with Lagrangian fibrations (with  discriminant  high enough and almost coprime to the rank). The main points in the proof of the latter result are the following.
Let $\cF$ be a spherical vector bundle on an elliptic $K3$ surface $S$: the vector bundle $\cE_X$ on a (small) Lagrangian deformation $X$ of $S^{[n]}$ obtained by extension of $\cF[n]^{+}$ restricts to slope stable semi-homogeneous vector bundles on Lagrangian fibers parametrized by a large open subset of the base  (the complement has codimension at least $2$). Any slope stable vector bundle $\cE$ on $X$ with the same rank, $c_1$ and $c_2$ as $\cE_X$ restricts 
to a slope stable semi-homogeneous vector bundle on a general Lagrangian fiber. Any two simple semi-homogeneous vector bundles on an abelian variety   with the same rank and determinant are obtained one from the other via tensorization with a (torsion) line bundle. This, together with a monodromy argument, gives that $\cE_X$ and $\cE$ restrict to isomorphic vector bundles on a general Lagrangian fiber. Since the set of Lagrangian fibers for which 
the restriction of $\cE_X$ is slope stable has complement of codimension at least $2$,  one concludes that $\cE_X$ and 
$\cE$ are isomorphic.

\subsection{Outline of the paper}
\setcounter{equation}{0}

Sections~\ref{sec:isospettro} and~\ref{sec:fibratifond} are devoted to the computation of the discriminant of the vector bundle $\cF[n]^{+}$ on $S^{[n]}$, provided $\cF$ is a spherical vector bundle on the $K3$ surface $S$. Since $\PP(\cF[n]^{+})$ extends to all (small) deformations of $S^{[n]}$, one knows a priori that  the discriminant is a linear combination of 
$c_2(S^{[n]})$ and the inverse  $q^{\vee}_{S^{[n]}}$ of the  BBF quadratic form. From this it follows that one can work on the open subset of $S^{[n]}$ parametrizing subschemes whose support has cardinality at least $n-1$, and then a straightforward computation gives that the discriminant is as in~\eqref{chernvbe}. 

In Section~\ref{sec:hilbdielle} we show that by starting from  slope stable spherical vector bundles $\cF$ on an elliptic surface $S\to\PP^1$ we can produce vector bundles $\cF[n]^{+}$ on $S^{[n]}$ with rank and first two Chern classes covering all the cases in Theorem~\ref{thm:unicita}. Moreover we study the restriction of such an $\cF[n]^{+}$ to  fibers of the Lagrangian fibration $S^{[n]}\to(\PP^1)^{(n)}\cong \PP^n$.

Section~\ref{sec:spagnesi}  is the most demanding part of the paper. The key ideas, outlined in Section~\ref{sec:ideebase}, are combined together in order  to give the proof of Theorem~\ref{thm:unicita} (and of Proposition~\ref{prp:rangociuno}).
\subsection{Acknowledgments}
\setcounter{equation}{0}
Thanks to Marco Manetti for explaining Horikawa's classical result~\cite[Theorem 6.1]{horikawa-II}, and to Emanuele Macr\`i for helping me out by introducing me to the powerful~\cite[Theorem~2]{simpshiggs}.
\section{The isospectral Hilbert scheme}\label{sec:isospettro}
\subsection{Summary of results}\label{subsec:finardi}
\setcounter{equation}{0}
We start by introducing  notation  and recalling known  results. 
Let $S$ be a $K3$ surface. The isospectral Hilbert scheme of $n$ points on $S$, denoted by $X_n=X_n(S)$, was introduced and studied by Haiman,  
 see Definition 3.2.4 in~\cite{haiman}. We have a commutative diagram
\begin{equation}\label{commiso}
\xymatrix{ X_n(S)\ar[d]_{\rho}\ar[rr]^{\tau}    &  &  S^n \ar[d]^{\pi}\\ 
  S^{[n]}  \ar[rr]^{\gamma} & & S^{(n)} }
\end{equation}
In fact $X_n(S)$ is the reduced scheme associated to the fiber product of $S^n$ and $S^{[n]}$ over $S^{(n)}$. 
Moreover the map $\tau$ is identified with the blow up of $S^n$ with center the big diagonal, see Corollary 3.8.3 in~\cite{haiman}. Let $\pr_i\colon S^n\to S$ be the $i$-th projection, and 
let $\tau_i\colon X_n(S)\to S$ be the composition $\tau_i:=\pr_i\circ \tau$. Let
\begin{equation}
(S^n)_{*}:=\{x=(x_1,\ldots,x_n)\in S^n \mid \text{at most two entries of $x$ are equal}\},
\end{equation}
and let $X_n(S)_{*}:=\tau^{-1}((S^n)_{*})$. Let 
$E_n\subset X_n(S)_{*}$ be the  exceptional 
divisor of $X_n(S)_{*}\to (S^n)_{*}$.
Then  $E_n$ is smooth because the restriction of the big diagonal to $(S^n)_{*}$ is smooth. We let
\begin{equation}
e_n:=\cl(E_n)\in H^2(X_n(S)_{*};\QQ).
\end{equation}
Let $\eta\in H^4(S;\QQ)$ be the fundamental class. 

If $X$ is a HK manifold, the non degenerate BBF symmetric blinear form $H^2(X)\times H^2(X)\to\CC$ defines a symmetric bilinear form 
$H^2(X)^{\vee}\times H^2(X)^{\vee}\to\CC$, i.e.~a symmetric element of $H^2(X)\otimes H^2(X)$, whose image in $H^4(X)$ via the cup product map is a rational Hodge class $q^{\vee}_X$.

Below are the results obtained in the present section. 
\begin{prp}\label{prp:ixnessen}
Let $n\ge 2$. We have the following equalities in the rational cohomology of $X_n(S)_{*}$:
\begin{eqnarray}
\rho^{*}\left(\ch_2(S^{[n]}\right)_{|X_n(S)_{*}} & = &-24\sum_{l=1}^n\tau_l^{*}(\eta)_{|X_n(S)_{*}}+3 e_n^2, \label{pullchernchar} \\
\rho^{*}(q^{\vee})_{|X_n(S)_{*}} & = & -(2n-24)\sum_{l=1}^n \tau_l^{*}(\eta)-\frac{4n-3}{2n-2}e_n^2. \label{pullcasimir}
\end{eqnarray}
\end{prp}
  Before stating the next result we note that while $q^{\vee}_X$ is a rational cohomology class,   $(2n-2)q^{\vee}_X$ lifts 
 to an integral class (uniquely because the group $H^*(S^{[n]};\ZZ)$ is torsion free by the main result 
in~\cite{markman-poisson}).
\begin{prp}\label{prp:inietta}
Let $n\ge 4$, and let $T_n\subset H^4(S^{[n]};\ZZ)$ be the saturation of the subgroup spanned by $c_2(S^{[n]})$ and $(2n-2)q_X^{\vee}$.  The map
\begin{equation}\label{mappatienne}
\begin{matrix}
T_n & \lra & H^4(X_n(S)_{*};\ZZ) \\
a & \mapsto & \rho^{*}(a)_{|X_n(S)_{*}}
\end{matrix}
\end{equation}
is injective.
\end{prp}

The proof of Propositions~\ref{prp:ixnessen} and~\ref{prp:inietta} are respectively in Subsections~\ref{subsec:berkeley} and~\ref{subsec:asiago}.

\subsection{Proof of Proposition~\ref{prp:ixnessen}}\label{subsec:berkeley}
\setcounter{equation}{0}
We start by recalling a couple of formulae. First suppose that $j\colon D\hra W$ is the embedding of a smooth divisor in a smooth ambient variety, and that $\cF$ is a sheaf on $D$. Then by Grothendieck-Riemann-Roch and the push-pull formula we have
\begin{equation}\label{chernpushforw}
\ch(j_{*}(\cF))=j_{*}(\ch(\cF))\cdot \Td(\cO_W(D))^{-1}=j_{*}(\ch(\cF))\cdot (1-\frac{\cl(D)}{2}+\frac{\cl(D)^2}{6}+\ldots).
\end{equation}
Next we recall how one computes the Chern classes of a blow up.  
Let $Z$ be a smooth variety, and let $Y\subset Z$ be a smooth subvariety of pure codimension $c$. Let $f\colon \wt{Z}\to Z$ be the blow of $Y$. Let $j\colon E\hra \wt{Z}$ be the inclusion of the exceptional divisor of $f$, and let $e\in H^2(\wt{Z};\QQ)$ be the class of $E$. If $\cN_{Y/Z}$ is the normal bundle of $Y$ in $Z$, then $E\cong\PP(\cN_{Y/Z})$, and the restriction of $\cO_{ \wt{Z}}(E)$ to $E$ is isomorphic to the tautological sub line bundle $\cO_E(-1)$. Let $\cQ$ be the quotient bundle on $E$, i.e.~the vector bundle fitting into the exact sequence
\begin{equation}
0\lra \cO_E(-1)\lra f_E^{*}\cN_{Y/Z}\lra \cQ \lra 0,
\end{equation}
where $f_E\colon E\to Y$ is the restriction of $f$. The differential of $f$ gives the exact sequence
\begin{equation}
0\lra \Theta_{\wt{Z}}\overset{df}{\lra} f^{*}(\Theta_Z) \lra j_{*}(\cQ) \lra 0.
\end{equation}
Taking Chern characters, and applying the formula in~\eqref{chernpushforw} to the inclusion $j\colon E\hra \wt{Z}$ and the sheaf $\cQ$ we get the formula
\begin{multline}\label{chernblow}
\ch(\wt{Z})=f^{*}(\ch(Z))-\ch( j_{*}(\cQ) )=f^{*}(\ch(Z))-j_{*}(\ch( (\cQ) )\cdot\left(1-\frac{e}{2}+\ldots\right)=\\
=f^{*}(\ch(Z))-j_{*}\left((c-1)+f_E^{*}(c_1(\cN_{Y/Z})-j^{*}(e)\right)\cdot\left(1-\frac{e}{2}+\ldots\right)\equiv\\
\equiv f^{*}(\ch(Z))-
j_{*}\left(f_E^{*}(c_1(\cN_{Y/Z})\right)-(c-1)e+\frac{c+1}{2}e^2\pmod{H^6(\wt{Z};\QQ)}
\end{multline}
\begin{proof}[Proof of the equality in~\eqref{pullchernchar}]
Since $X_n(S)_{*}$ is the blow up of $(S^n)_{*}$ with center the smooth locus of the big diagonal, we can relate the Chern characters of $X_n(S)_{*}$ and $(S^n)_{*}$  via the equality in~\eqref{chernblow}. Since $\ch_2(S^n)=-24\sum_{i=1}^n\tau_i^{*}(\eta)$, and the normal bundle of the big diagonal in $S^n$ has trivial first Chern class, the equation
 in~\eqref{chernblow}  gives that
\begin{equation}\label{granbret}
\ch_2(X_n(S)_{*})=\frac{3}{2}e_n^2 -24\sum_{l=1}^n\tau_l^{*}(\eta)_{|X_n(S)_{*}}.
\end{equation}
 The differential of the map $\rho\colon X_n(S)\to S^{[n]}$ gives the exact sequence
\begin{equation}
0\lra \rho^{*}(\Omega^1_{(S^{[n]})_{*}})\overset{(d\rho)^t}{\lra} \Omega^1_{X_n(S)_{*}}  \lra \iota_{*}(\iota^{*}(\cO_{X_n(S)_{*}}(-E_n))) \lra 0, 
\end{equation}
where $\iota\colon E_n\hra X_n(S)_{*}$ is the inclusion map.  Taking Chern characters we get that
\begin{equation}\label{alemagna}
\ch_2(X_n(S)_{*})=\rho^{*}\ch_2((S^{[n]})_{*})-\frac{3}{2}e_n^2.
\end{equation}
The equality in~\eqref{pullchernchar} follows from the equalities in~\eqref{granbret} and~\eqref{alemagna}.  
\end{proof}
\begin{proof}[Proof of  the equality in~\eqref{pullcasimir}]
\setcounter{equation}{0}
Let $\{\alpha_1,\ldots,\alpha_{22}\}$ be an orthonormal basis of $H^2(S;\CC)$. Then
\begin{equation}
q^{\vee}=\sum_{i=1}^{22}\mu(\alpha_i)^2-\frac{1}{2n-2}\delta_n^2,
\end{equation}
and hence
\begin{multline}\label{mcqueen}
\rho^{*}(q^{\vee})_{|X_n(S)_{*}}=\sum_{i=1}^{22}\left(\sum_{l=1}^n \tau_l^{*}(\alpha_i)\right)^2-\frac{1}{2n-2}e_n^2=\\
=22\sum_{l=1}^n \tau_l^{*}(\eta)+2\sum_{1\le l<m\le n}\left(\sum_{i=1}^{22}\tau_l^{*}(\alpha_i)\cdot \tau_m^{*}(\alpha_i)\right)-\frac{1}{2n-2}e^2.
\end{multline}

Let $D_n\subset S^n$ be the big diagonal. Then
\begin{equation}\label{newman}
\cl(D_n)=(n-1)\sum_{l=1}^n \tau_l^{*}(\eta)+\sum_{1\le l<m\le n}\left(\sum_{i=1}^{22}\tau_l^{*}(\alpha_i)\cdot \tau_m^{*}(\alpha_i)\right).
\end{equation}
Moreover it follows from the \lq\lq Key Formula\rq\rq (see for example Proposition~6.7 in~\cite{fulton-intersection-theory})  that we have the relation 
\begin{equation}\label{redford}
\tau^{*}(\cl(D_n))_{|X_n(S)_{*}}=-e_n^2.
\end{equation}
The equality in~\eqref{pullcasimir} follows at once from the equalities in~\eqref{newman} and~\eqref{redford}.
\end{proof}
\subsection{Proof of Proposition~\ref{prp:inietta}}\label{subsec:asiago}
\setcounter{equation}{0}
\begin{lmm}\label{lmm:dueinter}
Let $n\ge 2$. Then the classes $\alpha_n:=\sum_{l=1}^n\tau_l^{*}(\eta)_{|X_n(S)_{*}}$ and $e_n^2$ are linearly independent in $H^4(X_n(S)_{*};\ZZ)$. 
\end{lmm}
\begin{proof}
We prove the lemma  by integrating $\alpha_n$ and $e_n^2$ over algebraic $2$ cycles on $X_n(S)_{*}$ defined as follows. Let $p_1,\ldots,p_{n-1}\subset S$ be  $n-1$ distinct points, and let 
\begin{equation}
\Gamma:=\rho^{-1}(\{p_1,\ldots,p_{n-1},x) \mid x\in S\}).
\end{equation}
Clearly $\Gamma\subset X_n(S)_{*}$, and it is isomorphic to the blow up of $S$ at  $p_1,\ldots,p_{n-1}$. 
In order to define the second $2$ cycle we assume (as we may) that $S$ 
 contains two smooth curves $C_1,C_2$ intersecting  with  transverse intersection of cardinality $d>0$.
Let $q_1,\ldots,q_{n-2}\subset (S\setminus C_1\setminus C_2)$ be  $n-2$ distinct points, and let 
\begin{equation}
\Omega:=\rho^{-1}(\{q_1,\ldots,q_{n-2},x_1,x_2) \mid x_i\in C_i\}).
\end{equation}
Clearly $\Omega\subset X_n(S)_{*}$, and it is isomorphic to the blow up of $C_1\times C_2$ at  the $d$ points $(x,x)$ for $x\in C_1\cap C_2$. 

It makes sense to integrate $\alpha_n$ and $e_n^2$ over $\Gamma,\Omega$ 
because the latter are compact (complex) surfaces contained in $X_n(S)_{*}$.  
One checks easily that the $2\times 2$ \lq\lq Gram matrix\rq\rq\  of the integrals of $\alpha_n$ and $e_n^2$ over $\Gamma$ and $\Omega$ is given by Table~\ref{matgram}.
\begin{table}[tbp]
\caption{Integrals of $\alpha_n,e_n^2$ over $\Gamma,\Omega$}\label{matgram}
\vskip 1mm
\centering
\renewcommand{\arraystretch}{1.60}
\begin{tabular}{l l l}
   &  $\alpha_n$ & $e_n^2$ \\
\midrule
 $\Gamma$  &  $1$ & $-(n-1)$ \\
\midrule
$\Omega$ & $0$ & $-d$ \\
\bottomrule 
\end{tabular}
\end{table} 
It follows 
that $\alpha_n$ and $e_n^2$ are linearly independent. 
\end{proof}
Now we can prove Proposition~\ref{prp:inietta}. Proposition~\ref{prp:ixnessen} expresses the restriction to $X_n(S)_{*}$ of  
 $\rho^{*}(c_2(S^{[n]})$ and $q^{\vee}$ as linear combinations of $\alpha_n$ and $e_n$. The determinant of the $2\times 2$ matrix with entries the corresponding coefficients is non singular if and only if $n\notin\{2,3\}$, hence Proposition~\ref{prp:inietta} follows from  Lemma~\ref{lmm:dueinter}. 
\begin{rmk}
Let $n\in\{2,3\}$. By Proposition~\ref{prp:ixnessen} the classes $\rho^{*}\left(\ch_2(S^{[n]}\right)_{|X_n(S)_{*}}$ and 
$\rho^{*}(q^{\vee})_{|X_n(S)_{*}}$ are linearly dependent. This agrees with known results. In fact if $X$ is a HK of type $K3^{[2]}$ then 
 $c_2(X)$ and $q^{\vee}_X$ are linearly dependent because $\Sym^2 H^2(X;\QQ)=H^4(X;\QQ)$, and if $X$ is a HK of type $K3^{[3]}$ then $c_2(X)$  and $q^{\vee}_X$ are linearly dependent although $\Sym^2 H^2(X;\QQ)$ is strictly contained in  $H^4(X;\QQ)$, see Example~14 in~\cite{markman:coh-mod-sympl}, or Remark~3.3 in~\cite{green-kim-laza-robles}.

\end{rmk}

\section{Basic modular vector bundles on $S^{[n]}$}\label{sec:fibratifond}
\subsection{Summary of results}
\setcounter{equation}{0}
Let $S$ be a $K3$ surface. We maintain the notation introduced in Subsection~\ref{subsec:finardi}. Let $\cF$ be a locally free sheaf on $S$. Then
\begin{equation}
X_n(\cF):=\tau_1^{*}(\cF)\otimes\ldots\otimes\tau_n^{*}(\cF)
\end{equation}
is a locally free sheaf on $X_n(S)$. The map $\rho$ in~\eqref{commiso} is finite,  and moreover it is flat because 
$X_n(S)$ is CM by Theorem~3.1  in~\cite{haiman}. It follows that the pushforward $\rho_{*}(X_n(\cF))$ is also locally free. 
The  symmetric group $\cS_n$ acts on $X_n(S)$ compatibly with its permutation action on $S^n$, and hence 
the action lifts to  an action  $\mu_n^{+}$ on  $X_n(\cF)$.  Since $\mu_n$  maps to itself any fiber of  $\rho\colon X_n(S)\to S^{[n]}$,  we get
 an action  $\ov{\mu}_n^{+}\colon \cS_n \to \Aut(\rho_{*}X_n(\cF))$. 
\begin{dfn}
Let $\cF[n]^{+}\subset \rho_{*}X_n(\cF)$ be the sheaf of $\cS_n$-invariants for $\ov{\mu}_n^{+}$.
\end{dfn}
Since $\rho_{*}X_n(\cF)$ is locally free, so is $\cF[n]^{+}$. 

Let $r_0$ be the rank of $\cF$. Below is the main result of the present section.
\begin{prp}\label{prp:primeclassi}
Suppose that $\cF$ is spherical, i.e.~$h^p(S,\End^0(\cF))=0$ for all $p$. Let $n\ge 2$. Then
\begin{eqnarray}
\rk(\cF[n]^{+}) & = & r_0^n, \label{cizero} \\
c_1(\cF[n]^{+}) & = & r_0^{n-1} \left(\mu(c_1(\cF))-\frac{r_0- 1}{2}\delta_n\right), \label{ciuno} \\
\Delta(\cF[n]^{+})  & = & \frac{r_0^{2n-2}(r_0^2-1)}{12}c_2(S^{[n]}). \label{cidue}
\end{eqnarray}
\end{prp}
\begin{rmk}
The notation in~\eqref{ciuno} and~\eqref{cidue} is unambiguous because the group $H^*(S^{[n]};\ZZ)$ is torsion free by the main result 
in~\cite{markman-poisson} (see also~\cite{totaro-cohomhilb}). Notice that Proposition~\ref{prp:primeclassi} holds also (trivially) for $n=1$ provided we set $\delta_1=0$.
\end{rmk}
If $\cF$ is spherical,  then the vector bundle $\cF[n]^{+}$ is modular by Proposition~\ref{prp:primeclassi}, and we refer to it as a basic modular vector bundle. 
The proof of Proposition~\ref{prp:primeclassi} is in Subsection~\ref{subsec:dimoclassi}.
\begin{rmk}
The equalities in  Proposition~\ref{prp:primeclassi} should hold with the weaker hypothesis $\chi(S,\End^0(\cF))=0$. To prove this it would suffice to show that such a vector bundle is the limit of spherical vector bundles. 
\end{rmk}
\subsection{Chern classes of $\rho^{*}\cF[n]^{+}$ restricted to $X_n(S)_{*}$}
\setcounter{equation}{0}
Let $h_{+}\in H^{1,1}_{\QQ}(S^{[n]})$ be given by
\begin{equation}\label{deteffe}
h_{+}:=\mu(c_1(\cF))-\frac{r_0- 1}{2}\delta_n.
\end{equation}
In the present subsection we prove the following result.
\begin{prp}\label{prp:tiroclassi}
Let  $S$ be a $K3$ surface, and let $\cF$ be a vector bundle on $S$ such that $\chi(S,End^0 (\cF))=2$.
Let $r_0$ be the  rank of $S$, and   let $h^{+}\in H^{1,1}_{\QQ}(S^{[2]})$ be  as in~\eqref{deteffe}.
Then the following equalities hold:
\begin{eqnarray}
\ch_0(\rho^{*}\cF[n]^{+})_{|X_n(S)_{*}} & = & r_0^n \label{parzero} \\
\ch_1(\rho^{*}\cF[n]^{+})_{|X_n(S)_{*}} & = & r_0^{n-1} \rho^{*} (h^{+})_{|X_n(S)_{*}} \label{paruno} \\
\ch_2(\rho^{*}\cF[n]^{+})_{|X_n(S)_{*}} & = & r_0^{n-2}\rho^{*}
\left(\frac{(r_0^2-1)}{24}\ch_2(S^{[n]})+\frac{1}{2}h_{+}^2\right)_{|X_n(S)_{*}}  \label{pardue}
\end{eqnarray}
\end{prp}
\begin{proof}
Let $D_n\subset (S^n)_{*}$ be the (intersection of $(S^n)_{*}$ with the) big diagonal. For $1\le j<k\le n$ let $D_n(j,k)\subset D_n$ be the set of points 
$(x_1,\ldots,x_n)$ such that $x_j=x_k$. We have the open embedding
\begin{equation}\label{evidcoinc}
\begin{matrix}
D_n(j,k) & \overset{\ov{\epsilon}_{j,k}}{\hra} & S^{n-1} \\
(x_1,\ldots,x_n) & \mapsto & (x_j,x_1,\ldots,x_{j-1},\wh{x_j},x_{j+1},\ldots,x_{k-1},\wh{x_k},x_{k+1},\ldots,x_n)
\end{matrix}
\end{equation}
 Let  $\tau_{E_{n}}\colon E_{n}\to D_n$ be the restriction of $\tau$ to $E_{n}$, and let
  $E_n(j,k):=\tau_{E_{n}}^{-1}(D_n(j,k))$. 
Then $E_{n}=\coprod E_n(j,k)$. Let $\tau_{j,k}\colon E_n(j,k)\to D_n(j,k)$ be defined by the restriction of $\tau_{E_{n}}$, and let
\begin{equation}
\begin{matrix}
E_n(j,k) & \overset{\epsilon_{j,k}}{\lra} & S^{n-1} \\
y & \mapsto & \ov{\epsilon}_{j,k}(\tau_{j,k}(y))
\end{matrix}
\end{equation}
Let $\cQ_{j,k}$ be the locally free sheaf on 
$E_n(j,k)$  defined by 
\begin{equation}
\cQ_{j,k}:=\epsilon_{j,k}^{*}\left(\bigwedge^2\cF\boxtimes \cF\boxtimes\ldots\boxtimes\cF\right).
\end{equation}
Let $\iota_{j,k}\colon E_n(j,k)\hra X_n(S)_{*}$ be the inclusion map. We have an exact sequence
\begin{equation}
0\lra \rho^*\cF[n]^{+}\lra \tau_1^*(\cF)\otimes\ldots\otimes \tau_n^*(\cF)  \lra  
\oplus_{1\le j<k\le n} \iota_{j,k,*}(\cQ_{j,k})\lra 0.
\end{equation}
It follows that
\begin{equation}\label{tibbs}
 \rho^*\ch(\cF[n]^{+})= \tau_1^*\ch(\cF)\cdot\ldots\cdot \tau_n^*\ch(\cF) -  \sum_{1\le j<k\le n} \ch(\iota_{j,k,*}(\cQ_{j,k})).
\end{equation}
Since $\chi(S,End^0 (\cF))=2$, the
 Hirzebruch-Riemann-Roch Theorem gives that
\begin{equation}\label{sottsass}
2r_0\ch_2(\cF)=\ch_1(\cF)^2-2(r_0^2-1)\eta.
\end{equation}
Using the above equality, one gets that
modulo $H^6(X_n(S)_{*};\QQ)$ we have 
\begin{multline}\label{calda}
 \tau_1^*(\ch(\cF))\cdot\ldots\cdot \tau_n^*(\ch(\cF))= r_0^n+r_0^{n-1}\sum_{l=1}^n\tau_{l}^{*}(c_1(\cF))+ \\
+\frac{1}{2} r_0^{n-2}\sum_{l=1}^n\tau_{l}^{*}(c_1(\cF)^2-2(r_0^2-1)\eta)+\\
+ r_0^{n-2}\sum_{1\le l<m\le n}\tau_{l}^{*}(c_1(\cF))\cdot \tau_{m}^{*}(c_1(\cF)).
\end{multline}
Let $e_n(j,k)=\cl(E_n(j,k))$. 
Using the equality in~\eqref{chernpushforw}, one gets that modulo $H^6(X_n(S)_{*};\QQ)$  we have
\begin{multline*}
 \ch(\iota_{j,k,*}(\cQ_{j,k})) = \iota_{j,k,*}(\ch(\cQ_{j,k}))\cdot \left(1-\frac{e_n(j,k)}{2}\right)=\\
= \iota_{j,k,*}\left(\iota_{j,k}^{*}\left({r_0\choose 2}r_0^{n-2}+\frac{1}{2}(r_0-1)r_0^{n-2}\sum_{l=1}^n\tau_l^{*}c_1(\cF) \right)\right)
\cdot \left(1-\frac{e_n(i,j)}{2}\right)=\\
= {r_0\choose 2}r_0^{n-2} e_n(j,k)+\frac{1}{2}(r_0-1)r_0^{n-2} e_n(j,k)\cdot \sum_{l=1}^n\tau_l^{*}c_1(\cF)-
 \frac{1}{2}{r_0\choose 2}r_0^{n-2} e_n(j,k)^2.
\end{multline*}
The equalities in~\eqref{parzero}, \eqref{paruno} 
and~\eqref{pardue} follow at once from the equalities in~\eqref{tibbs}, \eqref{calda} and the above equality. 
\end{proof}
\subsection{Deformations of $(S^{[n]},\PP(\cF[n]^{+}))$}\label{subsec:deformocoppie}
\setcounter{equation}{0}
Let $\cF$ be a vector bundle on $S$, and let $f\colon\PP(\cF[n]^{+})\to S^{[n]}$ be the structure map. We let $\Def(\PP(\cF[n]^{+}),f, S^{[n]})$ be the deformation functor of the map $f$, see Definition 8.2.7 in~\cite{manetti-libro}.
\begin{prp}\label{prp:horimarco}
Suppose that the $K3$ surface $S$ is projective and that $\cF$ is a spherical vector bundle on $S$. Then the natural map 
$\Def(\PP(\cF[n]^{+}),f, S^{[n]})\to \Def(S^{[n]})$ is smooth.
\end{prp}
\begin{proof}
The result follows from a Theorem of Horikawa. In fact let $X:=\PP(\cF[n]^{+})$, $Y:=S^{[n]}$, and consider the exact sequence of locally free sheaves on $X$
\begin{equation}\label{succdiff}
0\lra \Theta_{X/Y}\lra \Theta_X\overset{df}{\lra} f^{*}\Theta_Y\lra 0.
\end{equation}
By~\cite[Theorem 6.1]{horikawa-II} (see also Corollary 8.2.14 in~\cite{manetti-libro}) it suffices to prove that 
the map  $H^1(X,\Theta_X)\to H^1(X,f^{*}\Theta_Y)$ is surjective and the map  $H^2(X,\Theta_X)\to H^2(X,f^{*}\Theta_Y)$ is injective. By the exact sequence in~\eqref{succdiff} it suffices to show that $H^2(X,\Theta_{X/Y})=0$. 
By the local-to-global spectral sequence abutting to $H^2(X,\Theta_{X/Y})$ we are done if we prove that 
\begin{equation}
H^p(X,R^q f_{*}(\Theta_{X/Y}))=0\qquad p+q=2.
\end{equation}
We have
\begin{equation}
R^q f_{*}(\Theta_{X/Y}))\cong 
\begin{cases}
End^0 \cF[n]^{+} & \text{if $q=0$,} \\
0 & \text{if $q>0$.} 
\end{cases}
\end{equation}
It follows from the McKay correspondence (see Proposition~5.4 in~\cite{ogfascimod}) that 
\begin{equation}\label{svanimento}
H^p(S^{[n]},End^0 \cF[n]^{+})=0\quad \forall p,
\end{equation}
and this finishes the proof.
\end{proof}
\begin{crl}\label{crl:deltasat}
Suppose that the $K3$ surface $S$ is projective and that $\cF$ is a spherical vector bundle on $S$. If $n\le 3$ then 
$\Delta(\cF[n]^{+})$ belongs to the saturation of $c_2(S^{[n]})$, if $n\ge 4$  then 
$\Delta(\cF[n]^{+})$ belongs to   the saturation of the span of $c_2(S^{[n]})$ and $q^{\vee}$.
\end{crl}
\begin{proof}
Let $\cX\overset{F}{\lra}\cY\overset{G}{\lra} T$ be representative of $\Def(\PP(\cF[n]^{+}),f, S^{[n]})$. Thus both $F$ and $G$ are proper holomorphic maps of analytic spaces, there exists $0\in T$ such that $F^{-1}(G^{-1}(0))\to G^{-1}(0)$ is identified with $f\colon \PP(\cF[n]^{+})\to S^{[n]}$, and every (small) deformation of $f$ is identified with 
$F^{-1}(G^{-1}(t))\to G^{-1}(t)$ for some $t\in T$ (close to $0$). For $t\in T$ (close to $0$) the map $F^{-1}(G^{-1}(t))\to G^{-1}(t)$ is a $\PP^{r-1}$ fibration, where $r=\rk(\cF[n]^{+})$, and hence the push-forward $F_{*}(\Theta_{\cX/\cY})$ is a vector bundle on $\cY$ (of rank $r^2-1$). 
By Proposition~\ref{prp:horimarco} the family $G\colon \cY\to T$ is versal at $t=0$, and hence the characteristic class 
$c_2(F_{*}(\Theta_{X_0/Y_0})$ (here $X_0=F^{-1}(G^{-1}(0))$ and $Y_0= G^{-1}(0)$) remains of type $(2,2)$ for all small deformation of $Y_0=S^{[n]}$. If $n\le 3$ it follows that $c_2(F_{*}(\Theta_{X_0/Y_0})$ belongs to the saturation of $c_2(S^{[n]})$, and if $n\ge 4$   it follows that $c_2(F_{*}(\Theta_{X_0/Y_0})$  belongs to the   saturation of the span of $c_2(S^{[n]})$ and $q^{\vee}$, see~\cite{zhang-char-form}. We are done because
\begin{equation*}
c_2(F_{*}(\Theta_{X_0/Y_0}))=c_2(End^0 \cF[n]^{+})=-\Delta(\cF[n]^{+}).
\end{equation*}
\end{proof}
\begin{rmk}\label{rmk:siestende}
Let $\Def(S^{[n]},\det\cF[n]^{+}))$ be the deformation functor of the couple $(S^{[n]},\det\cF[n]^{+}))$. The natural map  
$\Def(\cF[n]^{+})\to \Def(S^{[n]},c_1(\cF[n]^{+}))$ is an isomorphism, by the Artamkin-Mukai 
Theorem~\cite{muk-sympl,artamkin-deftns} (see also~\cite{iacono-man-deftn-pairs}) and by  
the vanishing in~\eqref{svanimento}. Hence $\cF[n]^{+}$ extends (uniquely) to a vector bundle on any small deformation  of $S^{[n]}$ keeping 
$c_1(\cF[n]^{+})$ of type $(1,1)$.  
\end{rmk}
\subsection{Proof of Proposition~\ref{prp:primeclassi}}\label{subsec:dimoclassi}
\setcounter{equation}{0}
The equality in~\eqref{cizero} follows at once from the equality in~\eqref{parzero}. Similarly, the equality in~\eqref{ciuno} follows at once from the equality in~\eqref{paruno}, because the restriction map 
$H^2(S^{[n]};\ZZ)\to H^2(\rho(X_n(S)_{*}))$ is an isomorphism (the complement of $\rho(X_n(S)_{*})$ in $S^{[n]}$ has codimension greater than one). Lastly we prove the equality in~\eqref{cidue}. Proposition~\ref{prp:tiroclassi} and a straightforward computation give that 
\begin{equation*}
\rho^{*}\Delta(\cF[n]^{+})_{|X_n(S)_{*}}=\rho^{*}\left(\frac{r_0^{2n-2}(r_0^2-1)}{12}c_2\left(S^{[n]}\right)\right)_{|X_n(S)_{*}}.
\end{equation*}
If $n\le 3$ then by Proposition~\ref{crl:deltasat} (note: we may assume that $S$ is projective) $\Delta(\cF[n]^{+})$ is a (possibly rational) multiple of $c_2\left(S^{[n]}\right)$. Since the restriction of 
$\rho^{*}c_2\left(S^{[n]}\right)$ to $X_n(S)_{*}$ is non zero (by the equality in~\eqref{pullchernchar} and 
Lemma~\ref{lmm:dueinter}), the equality in~\eqref{cidue} follows. If $n\ge 4$ then by Proposition~\ref{crl:deltasat} 
 $\Delta(\cF[n]^{+})$ is a linear combination (possibly with rational coefficients)  of $c_2\left(S^{[n]}\right)$ and $q^{\vee}$, and the equality follows from Proposition~\ref{prp:inietta}. 
\section{Basic modular vector bundles on $S^{[n]}$ for $S$ an elliptic $K3$ surface}\label{sec:hilbdielle}
\subsection{Contents of the section}
\setcounter{equation}{0}
We show that by starting from  slope stable spherical vector bundles $\cF$ on an elliptic surface $S$ we can produce vector bundles $\cF[n]^{+}$ on $S^{[n]}$ with rank and first two Chern classes covering all the cases in Theorem~\ref{thm:unicita}. We also study the restriction of such an $\cF[n]^{+}$ to  fibers of the Lagrangian fibration $S^{[n]}\to\PP^n$.
\subsection{Basic modular sheaves with the required topology}
\setcounter{equation}{0}
The present section contains analogues of the results in Sections 6.2 and 6.3 of~\cite{ogfascimod}.
Let $S$ be a $K3$ surface with an elliptic fibration $S\to \PP^1$; we let $C$ be a  fiber of the elliptic fibration.  The 
claim below follows from surjectivity of the period map for $K3$ surfaces.
\begin{clm}\label{clm:eccoell}
Let $m_0,d_0$ be positive natural numbers. There exist $K3$ surfaces $S$  with an elliptic fibration $S\to \PP^1$ such that 
\begin{equation}\label{neronsevero}
\NS(S)=\ZZ[D]\oplus\ZZ[C], \quad
D\cdot D=2m_0,\quad D\cdot C=d_0.
\end{equation}
\end{clm}
The following result is a (slight) extension of Proposition~6.2 in~\cite{ogfascimod} (and is more or less well known by experts). 
\begin{prp}\label{prp:rigsuk}
Let $m_0,r_0,s_0\in\NN_{+}$ and let $t,d_0\in\ZZ$. Suppose  that  
\begin{enumerate}
\item[(a)]
$t^2 m_0=r_0s_0-1$,
\item[(b)]
$d_0$ is  coprime to $r_0$, 
\item[(c)]
we have 
 \begin{equation}\label{digrande}
d_0>\frac{(2m_0+1)r_0^2(r_0^2-1)}{4}.
\end{equation}
\end{enumerate}
 Let $S$ be an elliptic $K3$ surface as in Claim~\ref{clm:eccoell}.
 Then there exists a vector bundle $\cF$ on $S$ such that the following hold:
 \begin{enumerate}
\item
 $v(\cF)=(r_0,tD,s_0)$,
\item
$\chi(S,End(\cF))=2$,
\item
$\cF$  is $L$ slope-stable for any polarization $L$ of $S$, 
\item
and the restriction of $\cF$ to \emph{every} fiber of  the  elliptic fibration $S\to \PP^1$ is slope-stable. 
\end{enumerate}
(Notice that every fiber is irreducible by our assumptions on 
$\NS(S)$, hence slope-stability of a sheaf on a fiber is well defined, i.e.~independent of the choice of a polarization.)
\end{prp}
\begin{proof}
One proceeds literally as in the proof of  Proposition~6.2 in~\cite{ogfascimod}.

\end{proof}
Assume that  $n,r_0,g,l,e\in\NN_{+}$, and that the equalities in~\eqref{nicoletta}, \eqref{nuvolari} and\eqref{cantarini}  (in Theorem~\ref{thm:unicita}) hold. Let
\begin{equation}\label{chartroux}
s_0:=\frac{g^2\ov{e}+ (2n-2)(r_0-1)^2+8}{8r_0},\quad m_0:=\frac{ g^2 \ov{e}+ 2(n-1)(r_0-1)^2}{8g^2 l^2}. 
\end{equation}
Then $s_0, m_0$ are  integers by the equalities in~\eqref{nuvolari} and in~\eqref{cantarini}. 
A straightforward computation gives that 
\begin{equation}\label{stanco}
(gl)^2 m_0=r_0s_0-1. 
\end{equation}
Let $S$ be a $K3$ surface as in Claim~\ref{clm:eccoell}, where $m_0$ is as in~\eqref{chartroux}, and $d_0$ is an integer  coprime to $r_0$  such that the inequality in~\eqref{digrande} holds. By Proposition~\ref{prp:rigsuk} 
there exists a  vector bundle $\cF$ on $S$ such that 
\begin{equation}\label{lezozzone}
v(\cF)=(r_0,glD,s_0)
\end{equation}
and Items (2)-(4) of that same proposition hold.
\begin{clm}\label{clm:classigiuste}
Keep notation and hypotheses as above, in particular $\cF$ is the vector bundle on $S$ such that the equation in~\eqref{lezozzone} 
and Items (2)-(4) of  Proposition~\ref{prp:rigsuk} hold. 
Let $\cE:=\cF[n]^{+}$. Then we have
 \begin{equation}\label{robertbyron}
r(\cE)=r_0^n,\quad c_1(\cE)=\frac{g\cdot r_0^{n-1}}{i} h,\quad \Delta(\cE)  =  
\frac{r_0^{2n-2}(r_0^2-1)}{12}c_2(S^{[n]}),
\end{equation}
where $h\in \NS(S^{[n]})$ is primitive, $q(h)=e$ and $\divisore(h)=il$.
\end{clm}
\begin{proof}
Let
\begin{equation}\label{nemequittepas}
h:=il\mu(\cl(D)-i\frac{r_0-1}{2g}\delta_n.
\end{equation}
Then $h$ is integral by the hypothesis in~\eqref{adancona}, primitive by the third equality in~\eqref{nicoletta}, $q(h)=e$ by the second equality in~\eqref{chartroux}, and $\divisore(h)=il$. The equalities in~\eqref{robertbyron} hold by Proposition~\ref{prp:primeclassi}.
\end{proof}
\subsection{Restriction of  $\cF[n]^{+}$ to Lagrangian fibers}\label{subsec:restringo}
\setcounter{equation}{0}
\begin{dfn}\label{dfn:laghilb}
Let $S\to\PP^1$ be an elliptically fibered $K3$ surface. If $x\in\PP^1$ we let $C_x$ be the (scheme theoretic) elliptic fiber over $x$.   Let $B=\{b_1,\ldots,b_m\}\subset\PP^1$ be the set of $x$ such that  $C_x$  is singular. Then $B$ is not empty (generically $m=24$). 
 The \emph{Lagrangian fibration associated to $S\to\PP^1$}  is the map
$\pi\colon S^{[n]}\to  \PP^n$  given by the  composition 
\begin{equation}\label{fiblagk3n}
S^{[n]}\to S^{(n)}\to (\PP^1)^{(n)}\cong \PP^n.
\end{equation}
\end{dfn}
We record a few facts regarding the (scheme theoretic) fibers of $\pi$. 
Let  $x_1,\ldots,x_n\in \PP^1$ be pairwise distinct: then 
\begin{equation}\label{prodfibre}
\pi^{-1}(x_1+\ldots+x_n)\cong C_{x_1}\times\ldots C_{x_n}.
\end{equation}
Next we describe the discriminant locus of $\pi\colon S^{[n]}\to  (\PP^1)^{(n)}$, i.e.~the subset  
$\cD\subset (\PP^1)^{(n)}$ parametrizing cycles $x_1+\ldots+x_n$ 
such that $\pi^{-1}(x_1+\ldots+x_n)$ is singular. For $b_j\in B$ let $\cD(b_j)\subset (\PP^1)^{(n)}$ be the irreducible divisor parametrizing 
cycles $x_1+\ldots+x_n$ such that $x_i=b_j$  for some $i\in\{1,\ldots,n\}$. Let
\begin{equation}
\cT:=\{\sum\limits_{i} m_i x_i\in (\PP^1)^{(n)}\mid \text{$m_i\ge 2$  for some $i\in\{1,\ldots,n\}$}\}.
\end{equation}
Note that the fibers of $\pi$ over points of $\cT$ are reducible and non reduced.
\begin{prp}
The  irreducible decomposition of the discriminant locus $\cD$ of $\pi\colon S^{[n]}\to  \PP^n$ is given by
\begin{equation}\label{decompdiscr}
\cD=\cT\cup\bigcup\limits_{j=1}^m \cD(b_j).
\end{equation}
\end{prp}
Below is the main result of the present subsection.
\begin{prp}\label{prp:acca20}
Let $S$ be a $K3$ surface with an elliptic fibration $S\to \PP^1$ as in Claim~\ref{clm:eccoell}, and let  
$\pi\colon S^{[n]}\to\PP^n$ be the associated Lagrangian fibration. Let $\cF$ be a vector bundle on $S$ as in Proposition~\ref{prp:rigsuk}. Then the following hold:
\begin{enumerate}
\item[(a)]
If   $x_1,\ldots,x_n\in \PP^1$ are pairwise distinct, then the restriction of $\cF[n]^{+}$ to
$\pi^{-1}(x_1+\ldots+x_n)$ is slope stable for any product polarization (this makes sense by the isomorphism in~\eqref{prodfibre}). 
\item[(b)]
Let $U\subset (\PP^1)^{(n)}$ be the open subset  parametrizing cycles
$x_1+\ldots+x_n$ such that the restriction of $\cF[n]^{+}$ to
$\pi^{-1}(x_1+\ldots+x_n)$ is a simple sheaf. The 
 complement  of $U$ has codimension at least two.
\end{enumerate}
\end{prp}
Before proving Proposition~\ref{prp:acca20}, we notice that Proposition~6.10 in~\cite{ogfascimod} holds for products of projective varieties of arbitrary dimension.
\begin{lmm}\label{lmm:prodstab}
For $i\in\{1,2\}$ let $(X_i,L_i)$ be an irreducible polarized projective variety of dimension $d_i$, and let $\cV_i$ be a slope stable vector bundle on   $X_i$. 
Then $\cV_1\boxtimes \cV_2$ is slope stable for any product polarization $\cL:=L_1^{m_1}\boxtimes L_2^{m_2}$ (of course $m_1,m_2\in\NN_{+}$).
\end{lmm}
\begin{proof}
Suppose that there exists an injection $\alpha\colon \cE\to \cV_1\boxtimes \cV_2$  with  torsion free cokernel such that $0<r(\cE)<r(\cV_1\boxtimes \cV_2)$ and
\begin{equation}\label{destabo}
\mu_{\cL}(\cE)\ge \mu_{\cL}(\cV_1\boxtimes\cV_2).
\end{equation}
The open subset $U\subset X_1\times X_2$ of points $p$ at which $\alpha$ is an injection of vector bundles (i.e.~the stalk of $\cE$ at $p$  is free and   $\alpha$ defines an injection of the fiber of $\cE$ at  $p$ to the fiber  of 
$\cV_1\boxtimes \cV_2$ at $p$) has complement of codimension at least $2$. Let $p=(x_1,x_2)\in U$. 
The restrictions of $\alpha$ to $\{x_1\}\times X_2$ and  to $X_1\times \{x_2\}$    are generically injective maps of vector bundles. Let
\begin{equation*}
A_1:=m_1^{d_1-1} m_2^{d_2}\deg X_2 {d_1+d_2-1\choose d_2},\quad 
A_2:=m_1^{d_1} m_2^{d_2-1}\deg X_1 {d_1+d_2-1\choose d_1}.
\end{equation*}
We have
\begin{equation}\label{sanmartino}
\mu_{\cL}(\cE)=A_1\mu_{L_1}(\cE_{|{X_1}\times \{x_2\}})+A_2 \mu_{L_2}(\cE_{|\{x_1\}\times X_2}),
\end{equation}
and 
\begin{equation}\label{pendab}
\mu_{\cL}(\cV_1\boxtimes\cV_2)=
A_1 \mu_{L_1}(\cV_1)+ A_2\mu_{L_2}(\cV_2).  
\end{equation}
Since the restrictions of $\cV_1\boxtimes\cV_2$  to $X_1\times \{x_2\}$ and to $\{x_1\}\times X_2$ are isomorphic to the polystable vector bundles   $\cV_1\otimes_{\CC}\CC^{r(\cV_2)}$ and $\cV_1\otimes_{\CC}\CC^{r(\cV_1)}$ respectively,  it follows from~\eqref{destabo}, \eqref{sanmartino} and~\eqref{pendab}  that $\mu(\cE_{|X_1\times \{x_2\}})=\mu(\cV_1)$ and $\mu(\cE_{|\{x_1\}\times X_2})=\mu(\cV_2)$. In turn, these equalities give that there exist vector subspaces 
$0\not=W_i\subset \CC^{r(\cV_i)}$  such that 
$\cE_{|X_1\times \{x_2\}}=\cV_1\otimes_{\CC} W_2$ on $U\cap(X_1\times \{x_2\})$   and 
$\cE_{|\{x_1\}\times X_2}=W_1\otimes_{\CC}\cV_2$ on  $U\cap(\{x\}_1\times X_2)$. It follows that 
$\im\alpha_{|U}=(\cV_1\boxtimes\cV_2)_{|U}$. This is a contradiction.
\end{proof}

\begin{proof}[Proof of Proposition~\ref{prp:acca20}]
(a): Follows from the stability of the restriction of $\cF$ to any elliptic fiber of $S\to\PP^1$ (Proposition~\ref{prp:rigsuk}), 
and Lemma~\ref{lmm:prodstab}.

\noindent
(b): Let $V\subset (\PP^1)^{(n)}$ be the open subset parametrizing  cycles
$\Gamma:=d_1 x_1+\ldots+ d_m x_m$ such that $d_j\le 2$ for all $j\in\{1,\ldots,m\}$, and  $C_{x_j}$ is smooth if $d_j=2$. If $\Gamma$ is such a cycle then 
the restriction of $\cF[n]^{+}$ to
$\pi^{-1}(\Gamma)$ is a simple sheaf. In fact $\pi^{-1}(\Gamma)$ is a product of schemes 
${\bm C}_1\times\ldots\times {\bm C}_m$, where $ {\bm C}_j=C_{x_j}$ if $d_j=1$, while  if $d_j=2$ then $ {\bm C}_j$ is identified with the scheme theoretic fiber over $2x_j\in(\PP^1)^{(2)}$ of the Lagrangian fibration 
$S^{[2]}\to(\PP^1)^{(2)}$.  Moreover 
\begin{equation}\label{cimbolli}
\cF[n]^{+}_{|\pi^{-1}(\Gamma)}\cong \left(\cF_{|{\bm C}_1}\right)\boxtimes\ldots\boxtimes\left(\cF_{|{\bm C}_m}\right).
\end{equation}
If $d_j=1$ then $ {\bm C}_j=C_{x_j}$ and $\cF_{|{\bm C}_j}$ is simple by Proposition~\ref{prp:rigsuk}. 
If $d_j=2$ then $ {\bm C}_j$ is the scheme considered in~\cite[Sect.~6.4]{ogfascimod}, and  
$\cF_{|{\bm C}_j}$ is simple by~\cite[Proposition 6.13]{ogfascimod}. By the isomorphism in~\eqref{cimbolli} it follows that 
$\cF[n]^{+}_{|\pi^{-1}(\Gamma)}$ is simple. Since the complement of $V$ in $(\PP^1)^{(n)}$ is a (closed) subset of codimension at least two, this proves Item~(b).
\end{proof}
\begin{rmk}\label{rmk:halloween}
By Item~(a) of Proposition~\ref{prp:acca20} the restriction of $\cF[n]^{+}$ to a (singular) Lagrangian fiber $X_t$ parametrized by a general point  
$t\in\cD(b_j)$ (notation as in~\eqref{decompdiscr}) is slope stable for any product polarization. 
\end{rmk}

The following remarks place Item~(a) of Proposition~\ref{prp:acca20} in the context of known results. 
\begin{rmk}\label{rmk:stabonlagfib}
Let $X\to\PP^n$ be  a Lagrangian fibration of a  HK manifold.  
For $t\in\PP^n$ let $X_t:=\pi^{-1}(t)$ be the schematic fiber of $X\to\PP^n$ over $t$. If $X_t$ is smooth there exists an ample primitive
class $\theta_t\in H^{1,1}_{\ZZ}(X_t)$ such that  
the image of the restriction map $H^2(X;\ZZ)\to H^2(X_t;\ZZ)$ is contained in $\ZZ\theta_t$,   see~\cite{wieneck1}.  If $\cF$ is a sheaf on $X_t$ slope-(semi)stability of $\cF$ will always mean $\theta_t$ slope-(semi)stability.
\end{rmk}
\begin{rmk}\label{rmk:k3nppav}
Let $X\to\PP^n$ be  a Lagrangian fibration of a  HK manifold of type $K3^{[n]}$, and let $X_t$ be a smooth Lagrangian fiber.  
Then the primitive ample
class $\theta_t\in H^{1,1}_{\ZZ}(X_t)$ is a principal polarization of $X_t$, see ~\cite{wieneck1}. If $S^{[n]}\to\PP^n$ is the Lagrangian fibration in~\eqref{fiblagk3n}, and $\pi^{-1}(x_1+\ldots+x_n)\cong C_{x_1}\times\ldots C_{x_n}$ is a smooth Lagrangian fiber, then 
$\theta_{x_1+\ldots+x_n}$ is the  product principal polarization.
\end{rmk}
\section{Proof of Theorem~\ref{thm:unicita} and Proposition~\ref{prp:rangociuno}}\label{sec:spagnesi}
\subsection{Contents of the section}
\setcounter{equation}{0}
In the present section we prove the following two  statements.
\begin{prp}\label{prp:stabesiste}
Let $n,r_0,g,l,e,i$ be as in Theorem~\ref{thm:unicita}. There exists an irreducible component  $\cK^{il}_e(2n)^{\rm good}$ 
of $\cK^{il}_e(2n)$ such that for a  general polarized HK variety 
 $(X,h)$  parametrized by  $\cK^{il}_e(2n)^{\rm good}$  there exists an $h$ slope-stable vector bundle $\cE$ on $X$ such that the equalities in~\eqref{chernvbe} hold
and moreover $H^p(X,End^0(\cE))=0$ for all $p$. 
\end{prp}
\begin{prp}\label{prp:stabunico}
Let $n,r_0,g,l,e,i$ be as in Theorem~\ref{thm:unicita}. If   $[(X,h)]\in \cK^{il}_e(2n)^{\rm good}$ 
is a general point, then   there exists a unique $h$ slope-stable vector bundle $\cE$ on $X$ such that 
 the equalities in~\eqref{chernvbe} hold.
\end{prp}
The same exact argument given  in the \lq\lq Proof of Theorem~1.4\rq\rq\ on p.~30 of~
\cite{ogfascimod} shows that
Theorem~\ref{thm:unicita} follows from Propositions~\ref{prp:stabesiste} and~ \ref{prp:stabunico}.

In Subsections~\ref{subsec:esistestab} and~\ref{subsec:poladatto} we prove results that are used in the proof of Propositions~\ref{prp:stabesiste} and~\ref{prp:stabunico}. Proposition~\ref{prp:stabesiste} is proved in Subsection~\ref{subsec:dimoesiste}. Subsections~\ref{subsec:erwitt} and~\ref{subsec:chebello} contain  results that are used in the proof of Proposition~\ref{prp:stabunico}. 
Propositions~\ref{prp:stabunico} and~\ref{prp:rangociuno} are proved in Subsections~\ref{subsec:alfinlaprova} and~\ref{subsec:restrizioni} respectively.
\subsection{The relevant component of $\cK^{il}_e(2n)$, and Noether-Lefschetz divisors}\label{subsec:esistestab}
\setcounter{equation}{0}
Let $S$ be an elliptic $K3$ surface as in Claim~\ref{clm:classigiuste}, and let $C,D$ be divisor classes generating 
$\NS(S)$ as in loc.~cit.  Let $X_0=S^{[n]}$, and let 
\begin{equation}\label{semundici}
X_0\overset{\pi_0}{\lra}(\PP^1)^{(n)}=\PP^n
\end{equation}
be the Lagrangian fibration associated to the elliptic fibration of $S$, see~\eqref{fiblagk3n}. 
Let
\begin{equation}\label{tuttipazzipermary}
h_0:=il\mu(\cl(D))-i\frac{r_0-1}{2g}\delta_n,\quad f_0:=\mu(\cl(C))=c_1(\pi_0^{*}\cO_{\PP^n}(1)).
\end{equation}
\begin{dfn}\label{dfn:compbuona}
Let  $\cK^{il}_e(2n)^{\rm good}\subset  \cK^{il}_e(2n)$ be the irreducible component containing $[(S^{[n]},h_0)]$.  
\end{dfn}
Let $d_0=C\cdot D$ (see~\eqref{neronsevero}). The sublattice $\la f_0,h_0\ra\subset H^{1,1}_{\ZZ}(X_0)$ is saturated and
\begin{equation}\label{edizero}
q(f_0)=0,\quad q(h_0,f_0)=ild_0,\quad q(h_0)=e.
\end{equation}
(The last equality follows from~\eqref{chartroux}.) Let  $L_0,F_0$  be the line bundles on $X_0$ such that $c_1(L_0)=h_0$ and $c_1(F_0)=f_0$. 
\begin{dfn}\label{dfn:eccobi}
Let $\varphi\colon\cX\to B$ be an (analytic) contractible representative of the functor $\Def(X_0,L_0,F_0)$. 
Let  $0\in B$ be the base point, so that $X_0=\varphi^{-1}(0)\cong S^{[n]}$. For $b\in B$ we let $X_b:=\varphi^{-1}(b)$, and we let $L_b,F_b$ be the line bundles on $X_b$ which are deformations of $L_0,F_0$ respectively. We let $h_b:=c_1(L_b)$ and $f_b:=c_1(F_b)$. 
\end{dfn}
Our first observation is that if  $d_0$ is large enough, then $h_b$ is ample for a general $b\in B$. 
Before proving this, we recall the following elementary result.
\begin{lmm}[Lemma~4.3 in~\cite{ogfascimod}]\label{lmm:nocamere}
Let $(\Lambda,q)$ be a non degenerate rank $2$ lattice which represents $0$, and hence $\disc(\Lambda)=-d^2$ where $d$ is a strictly positive integer. Let $\alpha\in \Lambda$ be primitive isotropic, and complete it to a basis 
$\{\alpha,\beta\}$ such that $q(\beta)\ge 0$. If $\gamma\in\Lambda$ has strictly negative square (i.e.~$q(\gamma)<0$) then
\begin{equation}\label{menoenne}
 q(\gamma)\le -\frac{2d}{1+q(\beta)}.
\end{equation}
\end{lmm}
\begin{prp}\label{prp:ampioperbigen}
Keep notation as above, and assume that
\begin{equation}\label{chiarello}
ild_0>(n-1)^2(n+3)(e+1).
\end{equation}
Then $L_b$ is ample on $X_b$ for a general $b\in B$. 
\end{prp}
\begin{proof}
Let $b\in B$ be a very general point, in the sense that $\NS(X_b)=\la h_b,f_b\ra$. 
By the inequality in~\eqref{chiarello} and  Lemma~\ref{lmm:nocamere}  there are no  $\xi\in \NS(X_b)$ such that $-2(n-1)^2(n+3)\le q(\xi)<0$.
By~\cite[Corollary~2.7]{mongardi-walls} it follows that
 the  ample cone of $X_b$ is equal to the intersection of $\NS(X_b)$ and the positive cone (if $R$ is the integral generator of an extremal ray then, viewed by duality as an element of $H^2(K3^{[n]},\QQ)$, 
the multiple $2(n-1) R$ is integral because
 the divisibility of any element of $H^2(X_b,\ZZ)$ is a divisor of $2n-2$). Hence either $h_b$ or $-h_b$ is ample. By considering the limit case $b=0$ we get that $h_b$ is ample. This proves that $h_b$ is ample for $b$ very general. Since $h_b$ is ample for $b$ in the complement of an anaytic subset of $B$, we are done.
\end{proof}
Assume that the inequality in~\eqref{chiarello} holds. By Proposition~\ref{prp:ampioperbigen} we have 
the moduli map 
\begin{equation}\label{periodibi}
\begin{matrix}
B & \lra & \cK^{il}_e(2n) \\
b & \mapsto & [(X_b,h_b)]
\end{matrix}
\end{equation}
Note that the image of the above period map is  contained in a unique (irreducible) Noether-Lefschetz divisor in 
$\cK^{il}_e(2n)^{\rm good}$.
\begin{dfn}\label{dfn:ecconl}
If the inequality in~\eqref{chiarello} holds, we let  $\NL(d_0)\subset  \cK^{il}_e(2n)^{\rm good}$ be the unique  irreducible Noether-Lefschetz divisor containing 
the image of the moduli map in~\eqref{periodibi}.  
\end{dfn}
\begin{rmk}\label{rmk:spiegovu}
Let  $[(X,h)]\in\NL(d_0)$ be a general point. Then there exists a well defined rank two subspace $V\subset \NS(X)$ such that $V=\la h,f\ra$, where 
\begin{equation}\label{battleroyal}
q(h,f)=ild_0,\quad q(f)=0.
\end{equation}
 (For $[(X,h)]$ in a proper Zariski closed  subset of $\NL(d_0)$ there might be more than one such rank two subspace). For almost all choices (of $n,r_0,g,l,e,i$ and $d_0$) there is a unique class $f\in V$ such that $V=\la h,f\ra$ and the equalities in~\eqref{battleroyal} hold, while for special choices there are two such classes. If monodromy exchanges these two isotropic classes there is no intrinsic way of distinguishing them. Abusing language we will speak of \lq\lq the class $f$\rq\rq. If $(X,h)=(X_0,h_0)$ (recall that $X_0=S^{[n]}$ where $S$ is our elliptic $K3$ surface) then
 $f_0=c_1(\pi_0^{*}\cO_{\PP^n}(1))$, where $\pi_0$ is the
  Lagrangian fibration   given by in~\eqref{semundici}. 
By~\cite[Theorem~1.2]{matsushita-iso-div} it follows that if   $[(X,h)]\in\NL(d_0)$ is a general point then 
there exists a Lagrangian fibration  $\pi_X\colon\cX\to \PP^n$ such that $f=c_1(\pi_X^{*}\cO_{\PP^n}(1))$. 
\end{rmk}
\begin{prp}\label{prp:fibvettnl}
Keep notation as above, and assume that the inequality in~\eqref{chiarello} holds, and that $d_0$ is coprime to $r_0$. Let  
$\NL(d_0)\subset  \cK^{il}_e(2n)^{\rm good}$ be the Noether-Lefschetz divisor of Definition~\ref{dfn:ecconl}. 
There exist an open dense $\NL(d_0)^*\subset \NL(d_0)$ and for each $[(X,h)]\in\NL(d_0)^*$ a vector bundle $\cE_X$ on $X$ such that 
 \begin{equation}\label{soliteformule}
r(\cE_X)=r_0^n,\quad c_1(\cE_X)=\frac{g\cdot r_0^{n-1}}{i} h,\quad \Delta(\cE_X)  =  \frac{r_0^{2n-2}(r_0^2-1)}{12}c_2(X), 
\end{equation}
$H^p(X,End^0(\cE_X))=0$ for all $p$, and the restriction of $\cE_X$ to a general smooth fiber  of the Lagrangian fibration $X\to \PP^n$ (see Remark~\ref{rmk:spiegovu}) is slope stable.
\end{prp}
\begin{proof}
Let $\cE_0:=\cF[n]^{+}$ be the vector bundle on $X_0$ of Claim~\ref{clm:classigiuste}. Note that $c_1(\cE_0)=h_0$ by loc.cit. 
If $B$ is small enough, then by  Remark~\ref{rmk:siestende} the vector bundle $\cE_0$ on $X_0$ deforms uniquely to a vector bundle $\cE_b$ on $X_b$, and hence we get a vector bundle $\cE_X$ on $X$ for $[(X,h)]$  in a dense open subset  $\cU\subset\NL(d_0)$. The equations in~\eqref{soliteformule} hold by 
the equations in~\eqref{robertbyron}. Since  $H^p(X,End^0(\cE_0))=0$ for all $p$ (see~\eqref{svanimento}) it follows from upper semicontinuity of the dimension of cohomology sheaves that for $[(X,h)]$ in a smaller dense open subset  $\cU'\subset\cU$ we have
$H^p(X,End^0(\cE_X))=0$ for all $p$. Lastly, it follows from Item~(a) of Proposition and Remark~\ref{rmk:k3nppav} 
that   for $[(X,h)]$ in a smaller dense open subset  
$\cU''\subset\cU'$ the restriction of $\cE_X$ to a general fiber  of the Lagrangian fibration $X\to \PP^n$ (see Remark~\ref{rmk:spiegovu}) is slope stable. We set $\NL(d_0)^*:=\cU''$.
\end{proof}
\subsection{Suitable polarizations}\label{subsec:poladatto}
\setcounter{equation}{0}
 We recall that if $h$ is $a$-suitable (see~\cite[Definition~3.5]{ogfascimod}) and $\cE$ is a vector bundle  on $X$ with $a(\cE)\le a$ (see~(3.1.1) loc.cit.~for the definition of $a(\cE)$), then  slope stability of  the restriction of $\cE$ to a general Lagrangian fiber (there is a canonical choice of polarization of any smooth Lagrangian fiber, see Remark~\ref{rmk:stabonlagfib}) implies slope stability of $\cE$, and the following weak converse holds: 
 slope stability of $\cE$ implies that the restriction of $\cE$ to a general Lagrangian fiber is slope semistable. 
\begin{lmm}\label{lmm:opportuno}
Keep assumptions and notation as above, and let $a>0$. Suppose that  
\begin{equation}\label{castagne}
ild_0>a(e+1). 
\end{equation}
Let $[(X,h)]\in\NL(d_0)$ be a general point, and let $f\in V\subset\NS(X)$ be as in Remark~\ref{rmk:spiegovu}. 
Then there does not exist $\xi\in\NS(X)$ 	such that 
\begin{equation}\label{disegua}
-a\le q_X(\xi)<0, \qquad q_X(\xi,h)>0,\quad q_X(\xi,f)>0.
\end{equation}
\end{lmm}
\begin{proof}
Let $\la h,f\ra=V\subset \NS(X)$ be as in Remark~\ref{rmk:spiegovu}. 
Applying Lemma~\ref{lmm:nocamere} to $\Lambda:=V$, $\alpha=f$ and $\beta=h$ one gets that there are no 
$\xi\in V$ such that $-a\le q_X(\xi)<0$.
In particular if $\NS(X)=\la h,f\ra$ (as is the case for very general $[(X,h)]\in\NL(d_0)$), then there is no 
$\xi\in V$ such that the inequalities in~\eqref{disegua} hold. 

It follows that the set of $[(X,h)]\in\NL(d_0)$ for which \emph{there exists} $\xi\in\NS(X)$
such that the inequalities in~\eqref{disegua} hold is  
  the intersection of $\NL(d_0)$ with a finite union of Noether-Lefschetz 
 divisors in  $\cK^{il}_e(2n)$, each of which does not contain $\NL(d_0)$. In fact suppose that  the inequalities in~\eqref{disegua} hold.
 Let $D$ be the (finite) index of $\la h,f\ra\oplus (\la h,f\ra^{\bot}\cap \NS(X))$ in $\NS(X)$. It is crucial to note that $D$ has an upper bound which only depends on the discriminant of the restrictions of $q_X$  to $V$ (i.e.~$-(ild_0)^2$) and to $V^{\bot}$, and the latter has an upper bound depending only on 
 the discriminant of $V$ and the discriminant of 
 $q_X$ (i.e.~$2n-2$). 
Then
\begin{equation}
\xi=\frac{\xi_1}{D}+\frac{\xi_2}{D},\qquad \xi_1\in\la h,f\ra,\quad \xi_2\in\la h,f\ra^{\bot}.
\end{equation}
Moreover we have just proved that $\xi_2$ is non zero.  Since the restriction of $q_X$ to $h^{\bot}\cap\NS(X)$ is negative definite we get that 
$q(\xi_2)<0$. We also have $q(\xi_1)<0$
by the last two inequalities in~\eqref{disegua}.  

Hence by the first inequality in~\eqref{disegua}  there exists a positive $M$ independent of $(X,h)$ such that $-M\le q(\xi_2)<0$ (here 
it is crucial that  $D$ has an upper bound which only depends on $(ild_0)^2$) and $2n-2$). Hence the moduli point of $(X,h)$ belongs to  the intersection of $\NL(d_0)$ with a finite union of Noether-Lefschetz 
 divisors in  $\cK^{il}_e(2n)$, and none of them contains $\NL(d_0)$ because if $\rho(X)=2$ then $[(X,h)]$ is not contained in any of these
 Noether-Lefschetz 
 divisors.
\end{proof}
\begin{prp}\label{prp:piazzolla}
Keep notation as above, and assume that 
\begin{equation}\label{psicologa}
ild_0>\frac{r_0^{4n-2}(r_0^2-1)(n+3)(e+1)}{8}.
\end{equation}
Let $[(X,h)]\in \NL(d_0)$ (the inequality in~\eqref{psicologa} implies that the inequality in~\eqref{chiarello} holds, and hence  the Noether-Lefschetz divisor $\NL(d_0)\subset  \cK^{il}_e(2n)^{\rm good}$ is defined).  
If $\cE$ is a vector bundle on $X$ such that the equalities in~\eqref{chernvbe} hold, 
   then $h$  is $a(\cE)$-suitable (relative to the associated Lagrangian fibration 
 $\pi\colon\cX\to \PP^n$).
\end{prp}
\begin{proof}
If $\alpha\in H^2(X)$ we have
\begin{equation*}
\int_X c_2(X)\cdot \alpha^{2n-2}=6(n+3)(2n-3)!! q_X(a)^{n-1},
\end{equation*}
(the  equality above. can be obtained from the known Hirzebruch-Huybrechts-Riemann-Roch formula for HK manifolds of Type $K3^{[n]}$) gives that $d(\cE_X)=r_0^{2n-2}(r_0^2-1)(n+3)/2$ (the definition of $d(\cE_X)$ is in~\cite[(1.2.2)]{ogfascimod}), and hence
\begin{equation*}
a(\cE_X)=\frac{r_0^{4n-2}(r_0^2-1)(n+3)}{8}.
\end{equation*}
The inequality in~\eqref{psicologa} and  Lemma~\ref{lmm:opportuno} give that $h$  is $a(\cE)$-suitable.
\end{proof}
\subsection{Proof of Proposition~\ref{prp:stabesiste}}\label{subsec:dimoesiste}
\setcounter{equation}{0}
Keeping notation and assumptions as above. Assume in addition  
 that $d_0$ is coprime to $r_0$ and that the inequality in~\eqref{psicologa} holds. (Note that the set of such $d_0$  is infinite.) Let $[(X,h)]\in\NL(d_0)^*$, and let $\cE_X$ be  a vector bundle    on $X$  as in Proposition~\ref{prp:fibvettnl}. 
By Proposition~\ref{prp:piazzolla} the polarization $h$  is $a(\cE_X)$-suitable relative to the Lagrangian fibration 
 $\pi_X\colon\cX\to \PP^n$.
 By Proposition~\ref{prp:fibvettnl} the restriction of $\cE_X$ to a general fiber of  $\pi_X\colon\cX\to \PP^n$ is slope stable. Since  $h$  is $a(\cE_X)$-suitable, the vector bundle $\cE_X$ is $h$ slope stable by ~\cite[Proposition~3.6]{ogfascimod}. We have $H^p(X,End^0(\cE_X))=0$ for all $p$, and hence $\cE_X$ extends (uniquely) to all small deformations of $(X,\det\cE_X)$ by Remark~\ref{rmk:siestende}. Since $c_1(\cE_X)$ is a multiple of $h$, we get that $\cE_X$ extends  to a vector bundle $\cE'$ on a general deformation $(X',h')$ of $(X,h)$. By openness of slope stability $\cE'$ is slope stable and by upper semicontinuity of cohomology dimension  $H^p(X,End^0(\cE'))=0$ for all $p$. 
\subsection{Tate-Shafarevich twists}\label{subsec:erwitt}
\setcounter{equation}{0}
A basic example of Lagrangian fibration is obtained as follows. Let $(S,h_S)$ be a polarized $K3$ surface of genus $n$. Let $\cJ(S)$ be the moduli space of rank $0$ pure $\cO_S(1)$ semistable sheaves $\xi$ with $\chi(\xi)=1-n$, i.e.~sheaves with Mukai vector $(0,h_S,1-n)$. The generic point of 
$\cJ(S)$ is represented by $i_{*}\cL$, where $i\colon C\hra S$ is the inclusion of a smooth $C\in \cO_S(1)$, and $\cL$ is a line bundle of degree $0$. Suppose that all divisors in the complete linear system $|\cO_S(1)|$ are irreducible and reduced. 
Then  every semistable sheaf parametrized by  $\cJ(S)$ is stable, and  $\cJ(S)$ is  a HK projective variety of Type $K3^{[n]}$. Moreover  
the support map $\cJ(S)\to |\cO_S(1)|\cong \PP^n$ is a Lagrangian fibration. 

Let  
$\NL(d_0)\subset  \cK^{il}_e(2n)^{\rm good}$ be the Noether-Lefschetz divisor of Definition~\ref{dfn:ecconl}, and let  
 $[(X,h)]\in\NL(d_0)$ be a general point. Then the associated  Lagrangian fibration  $\pi\colon X\to\PP^n$  is
 related to a (general) moduli space $\cJ(S)$ via a Tate-Shafarevich twist. In order to be more precise, we recall a result of Markman. 
First, if  $[(X,h)]\in\NL(d_0)$ is a general point, then there is an associated  polarized $K3$ surface  $(S,D)$ of genus $n$, and  moreover  $(S,D)$ is  a general such polarized surface  - see~\cite[Subsection~4.1]{markman-lagr}.
\begin{prp}\label{prp:torcere}
Keep notation as above, and assume that the inequality in~\eqref{chiarello} holds. 
 Let $[(X,h)]\in\NL(d_0)$ be a general point, let $X\to\PP^n$ be the associated Lagrangian fibration, and let  $(S,D)$ be the associated  polarized $K3$ surface (which is a general polarized $K3$ surface of genus $n$). Then $ X\to\PP^n$ is isomorphic to a Tate-Shafarevich twist of 
$\cJ(S)\to |D|$ via an identification $\PP^n\overset{\sim}{\lra}|D|$. 
\end{prp}
\begin{proof}
Suppose first that $\rho(X)=2$. Then, as shown in the proof of Proposition~\ref{prp:ampioperbigen}, the  ample cone of $X$ is equal to the positive cone (because of the inequality in~\eqref{chiarello}), and hence every  bimeromorphic map $X\dra X'$, where $X'$ is a $HK$,  is actually an isomorphism. It follows that $X$ is isomorphic to a Tate-Shafarevich twist of $\cJ(S)\to |D|$ by Theorem~7.13 in~\cite{markman-lagr}. The result follows from this because the locus in $\NL(d_0)$ parametrizing $(X,h)$ such that $\rho(X)=2$ is dense.
\end{proof}
Let 
$\Pic^0(X/\PP^n)$  be the relative Picard scheme of the Lagrangian fibration $ X\to\PP^n$ (notice that all fibers of $X\to\PP^n$ are irreducible by Proposition~\ref{prp:torcere}).
Let $U\subset\PP^n$ be the open dense set of regular values of $X\to\PP^n$. 
If  $t\in U$, the fiber of $\Pic^0(X/\PP^n)\to\PP^n$ over $t$ is an abelian variety $A_t$ (of dimension $n$) and
the fundamental group $\pi_1(U,t)$ acts by monodromy on the subgroup $A_{t,tors}$ of torsion points. 
\begin{crl}\label{crl:modinv}
Keep hypotheses and notation as above, and suppose that 
 $V\subset A_t[r_0^n]$ is a coset (of a subgroup of $ A_t[r_0^n]$) of cardinality $r_0^{2n}$ invariant under the action of monodromy. 
 Then $V= A_t[r_0]$.
\end{crl}
\begin{proof}
Let $(S,D)$ be the polarized $K3$ surface of genus $n$ associated to $X$ following Markman. 
Let $\cJ(S)_0\subset \cJ(S)$ be the open dense subset of   smooth points (i.e.~smooth points of $\cJ(S)$ with surjective differential) of the map 
$\cJ(S)\to |D|$. 
By Proposition~\ref{prp:torcere}, $\Pic^0(X/\PP^n)\to\PP^n$ is isomorphic to $\cJ(S)_0\to |D|$, for a certain identification $\PP^n\overset{\sim}{\lra} |D|$. Under this identification $t\in\PP^n$ corresponds to a smooth curve $C\in |D|$, and 
  the corresponding Lagrangian fiber $A_t$   is the Jacobian of  $C$. Hence we have a natural isomorphism 
\begin{equation}\label{toromo}
H_1(C;\QQ)/H_1(C;\ZZ)\overset{\sim}{\lra}A_{t,tors},
\end{equation}
 and the identification is compatible with the monodromy actions.

First we prove the result under the assumption that $V$ is a subgroup $G$. By the structure theorem for finite
abelian groups $G\cong \ZZ/(d_1)\oplus\ldots \oplus \ZZ/(d_r)$, where $r\le 2n$ (because 
$A_t[r_0^n]\cong \ZZ/(r_0^n)^{\oplus 2n}$) and $d_i| r_0^n$ for all $i$.   The monodromy action on  $H_1(C;\ZZ)$ is transitive on non zero elements 
by~\cite[Proposition~4.4]{fmos1} (if $n\ge 3$, if $n=2$ transitivity is proved by hand).
It follows from the isomorphism in~\eqref{toromo} that $r=2n$ and $d_1=\ldots=d_r$. Thus $d_i=r_0$ for all $i\in\{1,\ldots,2n\}$ because  $|G|=r_0^{2n}$. This proves the result under the assumption that $V$ is a subgroup.

Now let $V$ be a translate of a group $G$. Then $G=\{a-b \mid a,b\in V\}$, and hence $G$ is also  invariant for the monodromy action. Thus $G= A_t[r_0]$ by what we have just proved, and the coset
$V$ gives a point of the quotient $A_t[r_0^{2n}]/A_t[r_0^n]\cong A_t[r_0^n]$ which is invariant for the monodromy action.  There is a  unique invariant element, namely $0$, because of the isomorphism in~\eqref{toromo} and the transitivity of the monodromy action on   non zero elements of 
$H_1(C;\ZZ)$. Hence $V=A_t[r_0]$.
\end{proof}
\subsection{More properties of $\cE_X$ for $[(X,h)]\in\NL(d_0)^*$  a general point}\label{subsec:chebello}
\setcounter{equation}{0}
The main result of the present subsection is the following improved version of Proposition~\ref{prp:fibvettnl}.
\begin{prp}\label{prp:stabincoduno}
Keep notation as in Subsection~\ref{subsec:esistestab}, and assume that  $d_0$ is coprime to $r_0$, and  that the inequality in~\eqref{psicologa} holds. 
There exist an open dense subset $\NL(d_0)^{**}\subset \NL(d_0)^*$ and for each $[(X,h)]\in\NL(d_0)^{**}$ an $h$ slope stable vector bundle $\cE_X$ on $X$ such that the following hold:
\begin{enumerate}
\item
The equalities in~\eqref{soliteformule} hold.
\item
$H^p(X,End^0(\cE_X))=0$ for all $p$. 
\item
There exists  an open $\cU_X\subset\PP^n$, whose complement has codimension at least $2$, such for every 
$t\in \cU_X$ the restriction of $\cE_X$ to the  fiber $X_t$ over $t$ of the Lagrangian fibration $X\to \PP^n$   
is slope stable for the restriction of $h$ to $X_t$.
\end{enumerate}
\end{prp}
Before proving Proposition~\ref{prp:stabincoduno} we need to go through a couple of results. The result below   for abelian surfaces is~\cite[Proposition~4.4]{ogfascimod}. Part of the proof below is literally taken from the proof of 
Proposition~4.4 in~\cite{ogfascimod}, but there is a crucial extra input (not needed in dimension $2$), namely~\cite[Theorem~2]{simpshiggs}.
\begin{prp}\label{prp:numeretti}
Let $(A,\theta)$ be a principally polarized abelian  variety of dimension $n$, and let $\cF$ be a 
  slope semistable vector bundle on $A$ such that $c_1(\cF)$ is a multiple of $\theta$ and
$\Delta(\cF)\cdot \theta^{n-2}=0$.
Then there exist integers $r_0,m,b_0$ with $r_0,m\in\NN_{+}$ and $\gcd\{r_0,b_0\}=1$ such that 
\begin{equation}\label{dueg}
r(\cF)=r_0^n m,\quad c_1(\cF)=r_0^{n-1} b_0 m \theta.
\end{equation}
If $\cF$ is  not $\theta$ slope stable, then we may assume that $m>1$.
\end{prp}
\begin{proof}
If $\cF$ is slope stable, then it is simple semi-homogeneous by~\cite[Proposition~A.2]{ogfascimod}, and hence we may write~\eqref{dueg} with $m=1$ by~\cite[Proposition~A.3]{ogfascimod}.

Suppose that $\cF$ is strictly $\theta$ slope semistable, i.e.~there exists a destabilizing   exact sequence of torsion free sheaves 
\begin{equation}\label{potter}
0\lra \cG\lra \cF\lra \cH\lra 0
\end{equation}
with $\cG$ slope-stable.   Arguing as in the proof of~\cite[Proposition~4.4]{ogfascimod} one proves that
\begin{equation}\label{arcacor}
c_1(\cG)=a\theta,\quad c_1(\cH)=b\theta,\quad \Delta(\cG)\cdot \theta^{n-2}= \Delta(\cH)\cdot \theta^{n-2}=0.
\end{equation}
Let us prove that $\cG$ and $\cH$ are locally free. Let $r:=r(\cF)$ and let $m_r\colon A\to A$ be the multiplication by $r$ map. Let $\cL$ be a line bundle on $A$ such that $c_1(\cL)=-ra\theta$, and let $\cE:=m_r^{*}(\cF)\otimes\cL$. Then 
$\cE$ is slope semistable (because $\cF$ is) and
\begin{equation*}
c_1(\cE)=0,\quad \Delta(\cE)\cdot \theta^{n-2}=0.
\end{equation*}
By~\cite[Theorem~2]{simpshiggs} it follows that every quotient of the (slope) Jordan-H\"older filtration of $\cE$ is locally free. Since $m_r^{*}(\cG)\otimes\cL$ is a polystable subsheaf of $\cE$ we get that it is locally free. Thus $m_r^{*}(\cG)$  is locally free and hence also $\cG$ is locally free. By the equalities in~\eqref{arcacor} we may iterate this argument to show that also $\cH$ is locally free.

It follows that if 
\begin{equation*}
0=\cG_0\subsetneq\cG_1\subsetneq\ldots\subsetneq\cG_m=\cF
\end{equation*}
is a (slope) Jordan-H\"older filtration  of $\cF$, then each quotient $\cQ_i:=\cG_i/\cG_{i-1}$ is a 
slope stable   locally free sheaf with 
\begin{equation}
\frac{c_1(\cQ_i)}{\rk(\cQ_i)}=\frac{c_1(\cF)}{r(\cF)},\quad \Delta(\cQ_i)\cdot \theta^{n-2}=0.
\end{equation}
Hence each $\cQ_i$ is simple semi-homogeneous by~\cite[Proposition~A.2]{ogfascimod}, and therefore
by~\cite[Proposition~A.3]{ogfascimod} 
(see also~\cite[Remark~7.13]{muksemi})  there exist coprime integers $r_i,b_i$, with $r_i>0$, such that 
$r(\cQ_i)=r_i^n$ and $c_1(\cQ_i)=r_i^{n-1}b_i\theta$. Let $i,j\in\{1,\ldots,m\}$; since the slopes of $\cQ_i$ and 
$\cQ_j$ are equal, we get that $b_i r_j=b_j r_i$. 
It follows that $r_i=r_j$ and $b_i=b_j$ because  $\gcd\{r_i,b_i\}=\gcd\{r_j,b_j\}=1$. Thus $r(\cF)=m r_0^n$ and $c_1(\cF)=m r^{n-1}_0 b_0\theta$ where $r_0=r_i$ and $b_0=b_i$ for all $i\in\{1,\ldots,m\}$. We have $m\ge 2$ because we assumed that $\cF$ is strictly slope semistable.
\end{proof}
\begin{crl}\label{crl:caravilla}
Let $(A,\theta)$ be a principally polarized abelian variety  of dimension $n$,  and let $\cF$ be a 
 $\theta$ slope-semistable vector bundle on $A$ such that 
$\Delta(\cF)\cdot \theta^{n-2}=0$.
If $r(\cF)=r_0^n$ and  $c_1(\cF)=r_0^{n-1} b_0 \theta$,
 where $r_0, b_0$ are \emph{coprime} integers, then
 $\cF$ is  $\theta$ slope-stable.
\end{crl}
\begin{proof}
By contradiction. Suppose that $\cF$ is not   $\theta$ slope-stable. By Proposition~\ref{prp:numeretti} we may write
$r(\cF)=s_0^n m$, $c_1(\cF)=s_0^{n-1} c_0 m\theta$ where $s_0,m,c_0$ are integers (with $s_0,m>0$), $s_0,c_0$ are coprime and $m>1$. It follows that $s_0 b_0=c_0 r_0$. Since $\gcd\{r_0,b_0\}=1$ and $\gcd\{s_0,c_0\}=1$, we get that $r_0=s_0$ and hence $m=1$. This is a contradiction.
\end{proof}
\begin{proof}[Proof of Proposition~\ref{prp:stabincoduno}]
Let $\varphi\colon\cX\to B$ be as in Definition~\ref{dfn:eccobi}.  Recall that $X_0=\varphi^{-1}(0)\cong S^{[n]}$, where $S$ is an elliptic $K3$ surface as in Claim~\ref{clm:classigiuste}. 
Let $\cE_0:=\cF[n]^{+}$ be the vector bundle on $X_0$ of Claim~\ref{clm:classigiuste}. 
If $B$ is small enough,  the vector bundle $\cE_0$ on $X_0$ deforms uniquely to a vector bundle $\cE_b$ on $X_b$,  hence we get a vector bundle 
$\cE_X$ on $X$ for $[(X,h)]$  in a dense open subset  $\cU\subset\NL(d_0)$. Moreover $\cE_X$ is $h$ slope stable and Items~(1) and~(2) of Proposition~\ref{prp:stabincoduno} hold, see the proof of Proposition~\ref{prp:stabesiste}. For $[(X,h)]\in \NL(d_0)^*$, where
  $\NL(d_0)^*\subset\cU$   is an open dense subset, the restriction of $\cE_X$ to a general smooth fiber  of the Lagrangian fibration $X\to \PP^n$ is slope stable, see  Proposition~\ref{prp:stabesiste}. 
 
Let  $\cV_0\subset\PP^n$ be the set of $t$ such that the restriction of $\cE_0$ to the fiber over $t$ of the Lagrangian fibration $S^{[n]}\to\PP^n$ is simple. By Item~(b) of Proposition~\ref{prp:acca20}, $\cV_0$ is an open subset whose complement has codimension (in  $\PP^n$) at least $2$.
It follows that there exists an    open dense subset $\NL(d_0)^*_s\subset \NL(d_0)^*$ with the following property:
  if $[(X,h)]\in \NL(d_0)^*_s$
  the subset $\cV_X\subset\PP^n$ parametrizing fibers $X_t$ of 
  the Lagrangian fibration $X\to \PP^n$
 such that the restriction of $\cE_X$ to $X_t$ is simple  has complement of codimension (in $\PP^n$) at least $2$.

 Let  $[(X,h)]\in  \NL(d_0)^*_s$. We claim that if $t\in\cV_X$ and $X_t$ is smooth,  then    the restriction  $\cE_{X|X_t}$ is slope stable. 
 In order to prove this, we start by noting that for any Lagrangian (scheme-theoretic) fiber $X_t$ we have  
\begin{equation}\label{endpiatto}
\int_{[X_t]}\Delta(\cE_{X|X_t})\cdot \left(h_{|X_t}\right)^{n-2}=0.
\end{equation}
In fact the above equality is an easy consequence of the modularity of $\cE_X$, see~\cite[Lemma~2.5]{ogfascimod}.  
Let $X_t$ be a general smooth Lagrangian fiber.  By Proposition~\ref{prp:fibvettnl}  the restriction of $\cE_X$ to  $X_t$  is slope stable,  hence 
$\cE_{X|X_t}$ is semi-homogeneous because of the equality in~\eqref{endpiatto},  see~\cite[Lemma~2.5]{ogfascimod}. 
  It follows that if  $t_0\in\cV_X$ and $X_{t_0}$ is smooth,  then $\cE_{X|X_{t_0}}$  is (simple) semi-homogeneous. To prove this, we introduce some notation. For $t\in\PP^n$ such that $X_t$ is smooth let
\begin{eqnarray*}
\Phi^0(\cE_{X|X_t}) & := & \{(x,[\xi])\in X_t\times \wh{X}_t \mid \exists\ T_x^{*}(\cE_{X|X_t})\overset{\sim}{\lra} (\cE_{X|X_t})\otimes\xi\}, \\
\Psi^0(\cE_{X|X_t}) & := & \{(x,[\xi])\in X_t\times \wh{X}_t \mid \Hom(T_x^{*}(\cE_{X|X_t}),(\cE_{X|X_t})\otimes\xi)\not=0\}, \\
\end{eqnarray*}
where $T_x\colon X_t\to X_t$ is translation by $x\in X_t$ (locally in $t$ we may assume that $X_t$ is a family of abelian varieties rather than 
 torsors over  abelian varieties).
Recall that $\cE_{X|X_t}$ is semi-homogeneous if and only if $\Phi^0(\cE_{X|X_t})$ has dimension at least $n$, and that if that is the case 
then the group $\Phi^0(\cE_{X|X_t})$ has pure dimension $n$. Let $X_t$ be a general smooth Lagrangian fiber, so that  $\cE_{X|X_t}$ is slope stable and semi-homogeneous. Then  $\Phi^0(\cE_{X|X_t})$ has pure dimension $n$, and moreover (by slope stability) 
$\Phi^0(\cE_{X|X_t})=\Psi^0(\cE_{X|X_t})$. By upper semicontinuity of cohomology dimension it follows that every irreducible component of 
$\Psi^0(\cE_{X|X_{t_0}})$ has dimension at least $n$. 
  Now $(0,[\cO_{X_{t_0}}])\in \Psi^0(\cE_{X|X_{t_0}})$ and, since $\cE_{X|X_{t_0}}$ is simple, every non zero homomorphism  
  $\cE_{X|X_{t_0}}\to \cE_{X|X_{t_0}}$ is an isomorphism. It follows that $\Phi^0(\cE_{X|X_{t_0}})$ has dimension at least $n$, and hence 
  $\cE_{X|X_{t_0}}$ is semi-homogeneous. By~\cite[Proposition~6.13]{muksemi} we get that $\cE_{X|X_{t_0}}$ is slope semistable (actually it is Gieseker stable, see Proposition~6.16  loc.cit.).
  Lastly we  prove that   $\cE_{X|X_{t_0}}$ is   slope-stable. Let  $\theta_{t_0}$ be the principal polarization of $X_{t_0}$, see Remark~\ref{rmk:stabonlagfib}. Since $q_{X_{t_0}}(h,f)=ild_0$ (see~\eqref{battleroyal}), we have  $h_{|X_{t_0}}=ild_0\theta_{t_0}$. Hence
\begin{equation}\label{villabardini}
r(\cE_{X|X_{t_0}})=r_0^n,\quad c_1(\cE_{X|X_{t_0}})=r_0^{n-1}gld_0\theta_{t_0}.
\end{equation}
By hypothesis $g$, $l$ and $d_0$ are coprime to $r_0$. By Corollary~\ref{crl:caravilla} we get  that   $\cE_{X|X_{t_0}}$ is   slope-stable.

We finish the proof by showing that if   $[(X,h)]\in  \NL(d_0)^*_s$ is general then 
the restriction of $\cE_X$  to a general singular Lagrangian fiber  is slope-stable.  Since $[(X,h)]\in  \NL(d_0)^*_s$ is general 
the discriminant divisor $\cD_X\subset\PP^n$ parametrizing
singular Lagrangian fibers of $X\to\PP^n$ is the dual of an embedded $K3$ surface $S\subset(\PP^n)^{\vee}$. In fact this holds by 
Proposition~\ref{prp:torcere}. Hence $\cD_X$ is an irreducible  divisor. Thus it suffices to prove that there exist $t\in\cD_X$ such that 
$\cE_{X|X_t}$ is slope stable (for the restriction of $h$ to $X_t$). This follows from Remark~\ref{rmk:halloween}  and openness of slope stability.
\end{proof}
\subsection{Proof of Proposition~\ref{prp:stabunico}}\label{subsec:alfinlaprova}
\setcounter{equation}{0}
The  key result is the following.
\begin{prp}\label{prp:sonoisom}
Keep notation as in Subsection~\ref{subsec:esistestab}.  Assume in addition that  $d_0$ is coprime to $r_0$, and  that the inequality in~\eqref{psicologa} holds. Let $[(X,h)]\in \NL(d_0)^{**}$ be a general point. Then (up to isomorphism) there exists one and only one $h$ slope stable vector bundle $\cE$ on $X$ such that the equalities in~\eqref{chernvbe} hold. 
\end{prp}
\begin{proof}
Let $\cE_X$ be a vector bundle on $X$ as in Proposition~\ref{prp:stabincoduno}, and let $\cE$ be an $h$ slope stable vector bundle  on $X$ such that the equalities in~\eqref{chernvbe} hold. We prove that $\cE_X$ and $\cE$ are isomorphic.

Let $\pi\colon X\to\PP^n$ be the associated Lagrangian fibration. By Proposition~\ref{prp:piazzolla} the polarization $h$ is $a(\cE)$-suitable, and hence the restriction of $\cE$ to a general Lagrangian fiber is slope semistable. We claim that if $X_t$ is a smooth Lagrangian fiber and $\cE_{|X_{t}}$ is   slope 
semistable, then it is  actually   slope stable. In fact this follows from Corollary~\ref{crl:caravilla} - the computations showing that the hypotheses of Corollary~\ref{crl:caravilla} are satisfied have already been done, see~\eqref{villabardini}. 
The upshot is that there exists a dense open subset $\cU_X^0\subset\PP^n$, contained in the set of regular values of the Lagrangian fibration, with the property that for all $t\in \cU_X^0$ the restrictions $\cE_{X|X_t}$ and $\cE_{|X_t}$ are simple semi-homogeneous vector bundles (since they are slope stable vector bundles and $\Delta(\cE_{X|X_t})\cdot\theta_t^{n-2}=\Delta(\cE_{|X_t})\cdot\theta_t^{n-2}=0$, they are semi-homogeneous by~\cite[Proposition~A.2]{ogfascimod}). 

We claim that if  $t\in \cU_X^0$ then $\cE_{X|X_t}$ and $\cE_{|X_t}$ are isomorphic. First note that, since they are simple semi-homogenous vector bundles with same rank and $c_1$,  the set
\begin{equation*}
V_t:=\{[\xi]\in X_t^{\vee} \mid \cE_{X|X_t}\cong (\cE_{|X_t})\otimes\xi\}
\end{equation*}
is not empty  by~\cite[Theorem~7.11]{muksemi}, and  it has cardinality $r_0^{2n}$ by Proposition~7.1 op.~cit. Note that $V_t\subset X_t[r_0^{n}]$ because $\cE_{X|X_t}$ and 
$\cE_{|X_t}$ have rank $r_0^n$ and isomorphic determinants. 
Next we claim that $V_t$ is invariant under the monodromy action of 
$\pi_1(\cU_X^0,t)$. In fact let
\begin{equation}
\cV:=\bigcup\limits_{t\in\cU_X^0}V_t.
\end{equation}
We show that the forgetful map
\begin{equation}\label{bregovic}
\cV\to \cU_X^0
\end{equation}
is a topological covering. Let $t_1\in \cU_X^0$, and let $\xi_{t_1}\in V_{t_1}$. Let $B\subset \cU_X^0$ be an open (in the classical topology) neighborhood of $t_1$ which is contractible. For each $t\in B$ let 
$\xi_t\subset X_t[r_0^{n}]$ be obtained from $\xi_{t_1}$ by parallel transport. We claim that
\begin{equation}\label{battisti}
\cE_{X|X_t}\cong \left(\cE_{|X_t}\right)\otimes \xi_t\qquad \forall t\in B.
\end{equation}
In fact the traceless endomorphism bundles of the right and left sides of~\eqref{battisti} have vanishing cohomologies, see Theorem~5.8 in~\cite{muksemi}. Since their determinants remain of type $(1,1)$ on $X_t$ for all $t\in B$ (actually for all $t\in\cU_X^0$), 
it follows that each of them  extends uniquely to all $X_{t'}$ for $t'\in B$ close enough to $t$. Since they are isomorphic for $t=t_1$, we get that they are isomorphic for all $t\in B$. This shows that the map in~\eqref{bregovic} is a topological covering.  The proof shows also that monodromy takes $V_t$ to itself.

Since $V_t$ is invariant under the monodromy action of 
$\pi_1(\cU_X^0,t)$, it follows from Corollary~\ref{crl:modinv} that $V_t=A[r_0]$. Thus $0\in V_t$, and therefore $\cE_{X|X_t}\cong \cE_{|X_t}$. 

Now  we use the hypothesis that  $[(X,h)]\in \NL(d_0)^{**}$ is a general point. Then the polarized $K3$ surface $(S,D)$ is a general surface of genus $n$ (see Subsection~\ref{subsec:erwitt}), and by Proposition~\ref{prp:torcere}  the Lagrangian fibration $\pi\colon X\to\PP^n$ is a Tate-Shafarevich twist of the relative Jacobian 
$\cJ(S)\to |D|$. It follows that the discriminant curve $B\subset\PP^n$ is isomorphic to the dual of $S\subset |D|^{\vee}$, and hence is reduced. Let 
$\cU_X^{\dagger}:=\cU_X\setminus \sing B$, where $\cU_X$ is as in Proposition~\ref{prp:stabincoduno}.   Note that $\cU_X^{\dagger}$ is an open subset 
of $\PP^n$ whose complement has codimension at least $2$. One proves that
\begin{equation}\label{sorbole}
\cE_{X|X_t}\cong \cE_{|X_t}\quad \forall t\in \cU_X^{\dagger}
\end{equation}
proceeding as in the proof of~\cite[Proposition~7.4]{ogfascimod}. More precisely there exists a smooth projective curve $T\subset\cU_X^{\dagger}$ containing $t$ and trasverse to $B$ (recall that the complement of $\cU_X^{\dagger}$ in $\PP^n$   has codimension at least $2$). Then $Y:=\pi^{-1}(T)$ is  a smooth projective (integral) variety of dimension $n+1$ and the sheaves $\cF:=\cE_{X|Y}$ and $\cG:=\cE_{|Y}$  satisfy the hypotheses of~\cite[Lemma~7.5]{ogfascimod}, and hence the isomorphism in~\eqref{sorbole} holds by the quoted lemma. 

Since $\cE_{X|X_t}$ and $\cE_{|X_t}$ are simple for all $t\in\cU_X^{\dagger}$, and since $c_1(\cE_X)=c_1(\cE)$, it follows that the restrictions of $\cE_X$ and $\cE$ to $\pi^{-1}(\cU_X^{\dagger})$ are isomorphic, see the proof of Proposition~7.4 in~\cite{ogfascimod}, in particular the beginning of the proof of Lemma~7.5. The complement of $\pi^{-1}(\cU_X^{\dagger})$ in $X$ has codimension at least $2$ because $\pi$ is equidimensional, and hence the isomorphism $\cE_{X|\pi^{-1}(\cU_X^{\dagger})}\overset{\sim}{\lra} \cE_{|\pi^{-1}(\cU_X^{\dagger})}$ extends to an isomorphism $\cE_X\overset{\sim}{\lra} \cE$.
\end{proof}
We are ready to prove Proposition~\ref{prp:stabunico}.
Let $n,r_0,g,l,e$ be as in Theorem~\ref{thm:unicita}. Since the result is trivially true fo $r_0=1$ we assume that $r_0\ge 2$.
Let $\cX\to T_{e}^{il}(2n)$  be a complete family of  polarized 
 HK varieties of Type $K3^{[n]}$ parametrized by $\cK_{e}^{il}(2n)^{\rm good}$. Since $\cK_{e}^{il}(2n)^{\rm good}$ is irreducible we may, and will, assume that $T_{e}^{il}(2n)$ is irreducible.  
 For $t\in T_{e}^{il}(2n)$ we let $(X_t,h_t)$ be the corresponding polarized HK of Type $K3^{[n]}$. Let 
 $m\colon T_{e}^{il}(2n)\to \cK_{e}^{il}(2n)^{\rm good}$ be the moduli map  sending $t$ to $[(X_t,h_t)]$.

By Gieseker and Maruyama there exists a  relative moduli space 
\begin{equation}\label{modrel}
\cM(r_0,g)\overset{f}{\lra} T_{e}^{il}(2n),
\end{equation}
 such that for every $t\in T_{e}^{il}(2n)$ the (scheme theoretic) fiber $f^{-1}(t)$  is isomorphic to the (coarse) moduli space  of 
  $h_t$ slope-stable vector bundles $\cE$ on $X_t$ such that~\eqref{chernvbe} holds. Moreover the morphism $f$ is  of finite type by Maruyama~\cite{marubound}, and hence  $f(\cM(r_0,g))$ is a  constructible subset of $T_{e}^{il}(2n)$.
  
Let $d(r_0,e,l)$ be the right hand side of the inequality in~\eqref{psicologa}. For $t$ in   a dense subset of  
$\bigcup\limits_{d> d(r_0,e,l)}m^{-1}(\NL(d)^{\text{good}})$ the preimage $f^{-1}(t)$ is a singleton by Proposition~\ref{prp:sonoisom}. Since 
 $\bigcup\limits_{d> d(r_0,e,l)}m^{-1}(\NL(d)^{\text{good}})$ is Zariski dense in $T_{e}^{il}(2n)$ (it is the union of an infinite collection of pairwsie distinct divisors), and since
$f(\cM(r_0,g))$ is a  constructible subset of $T_{e}^{il}(2n)$, it follows that for  general $t\in  T_{e}^{il}(2n)$ the fiber 
$f^{-1}(t)$ is a singleton.

Let $[\cE]$ be the unique point of $f^{-1}(t)$ for $t$ a generic point of $m^{-1}(\NL(d)^{\text{good}})$, where $d> d(r_0,e,l)$. 
Then $H^p(X_t, End^0(\cE))=0$ by 
 Proposition~\ref{prp:stabincoduno}. Hence the last sentence of Theorem~\ref{thm:unicita} follows from upper semicontinuity of cohomology.

\subsection{Proof of Proposition~\ref{prp:rangociuno}}\label{subsec:restrizioni}
\setcounter{equation}{0}
Since the natural morphism $\Def(X,\cF)\to\Def(X,h)$ is surjective, for $d_0\gg 0$ there exist extensions of $\cF$ to  polarized HK varieties $(Y,h)$ of type $K3^{[n]}$ with a Lagrangian fibration $\pi\colon Y\to \PP^n$ such that 
\begin{equation}
q_Y(h,f)=d_0\cdot \divisore(h).
\end{equation}
(As usual $f:=c_1(\pi^{*}\cO_{\PP^n}(1))$.) Let $X_t$ be a smooth (Lagrangian) fiber of $\pi$, and let $\theta_t$ be the principal polarization of $X_t$ induced by the Lagrangian fibration (see Remark~\ref{rmk:k3nppav}). We claim that
\begin{equation}\label{longton}
h_{|Y_t}=d_0\cdot \divisore(h) \,\theta_t.
\end{equation}
In fact the above equality follows from the equalities
\begin{equation*}
\int_{Y_t}(h_{|Y_t})^n=\int_{X}h^n\cdot f^n=n!\, q_X(h,f)^n=n! \, (d_0\cdot\divisore(h))^n.
\end{equation*}
By the equality in~\eqref{longton} we get that
\begin{equation*}
c_1({\cE|Y_t})=a\cdot d_0\cdot \divisore(h) \,\theta_t.
\end{equation*}
By Proposition~\ref{prp:numeretti} we may write
\begin{equation}
r(\cE)=r_0^n m,\quad a\cdot d_0\cdot \divisore(h)=r_0^{n-1} b_0 m,
\end{equation}
where $r_0,m,b_0$ are integers, $r_0,m\in\NN_{+}$ and $\gcd\{r_0,b_0\}=1$. 
Choose $d_0$ coprime to $r(\cE)$: then $d_0$ divides $b_0$ and the equation in~\eqref{thatsamore} holds with $b_0'=b_0/d_0$.

 \bibliography{ref-vbs-on-type-k3n}
 \end{document}